\documentclass{article}

\usepackage{arxiv}

\usepackage[utf8]{inputenc} 
\usepackage[T1]{fontenc}    
\usepackage{hyperref}       
\usepackage{url}            
\usepackage{booktabs}       
\usepackage{amsfonts}       
\usepackage{nicefrac}       
\usepackage{microtype}      
\usepackage{lipsum}		
\usepackage{graphicx}
\usepackage{natbib}
\usepackage{doi}

\usepackage{algorithm}
\usepackage{algorithmic}
\usepackage{enumitem}
\usepackage{wrapfig}

\usepackage{amsmath}
\usepackage{amssymb}
\usepackage{amsthm}
\newtheorem{theorem}{Theorem}

\newcommand{\reals}{\mathbb{R}}

\DeclareMathOperator*{\argmin}{arg\,min}
 
\newcommand{\E}[2]{\mathbb{E}_{#1}\!\left[#2\right]}
\newcommand{\Egiven}[3]{\mathbb{E}_{#1}\!\left[\left.\kern-\nulldelimiterspace#2\,\right|\,#3\right]}

\newcommand{\prob}[1]{\mathbb{P}\!\left(#1\right)}

\newcommand{\states}{\mathcal{S}}
\newcommand{\actions}{\mathcal{A}}
\newcommand{\st}[1]{\mathbf{s}_{#1}}
\newcommand{\at}[1]{\mathbf{a}_{#1}}

\title{Bounded-Error Policy Optimization for Mixed Discrete-Continuous MDPs via Constraint Generation in Nonlinear Programming}

\author{ Michael Gimelfarb \\
	Department of Computer Science \\
    University of Toronto, Toronto, Canada \\
	\texttt{mike.gimelfarb@mail.utoronto.ca} \\
	\And
	Ayal Taitler \\
	Department of Industrial Engineering and Management \\
	Ben Gurion University of the Negev, Be'er Sheva, Israel \\
	\texttt{ataitler@bgu.ac.il} \\
    \And
	Scott Sanner \\
	Department of Mechanical and Industrial Engineering\\
	University of Toronto, Toronto, Canada \\
	\texttt{ssanner@mie.utoronto.ca}
}

\renewcommand{\shorttitle}{Constraint-Generation Policy Optimization}

\hypersetup{
pdftitle={Constraint-Generation Policy Optimization},
pdfsubject={cs.SY, cs.SC},
pdfauthor={Michael Gimelfarb, Ayal Taitler, Scott Sanner},
pdfkeywords={Planning, Control, Mixed Discrete-Continuous MDP, Constraint Generation, Chance Constraints, Piecewise-Linear Policy, Sequential Decision Optimization, Policy Optimization},
}

\begin{document}
\maketitle

\begin{abstract}
	We propose the Constraint-Generation Policy Optimization (CGPO) framework to optimize policy parameters within compact and interpretable policy classes for mixed discrete-continuous Markov Decision Processes (DC-MDP). CGPO can not only provide bounded policy error guarantees over an infinite range of initial states for many DC-MDPs with expressive nonlinear dynamics, but it can also provably derive optimal policies in cases where it terminates with zero error. Furthermore, CGPO can generate worst-case state trajectories to diagnose policy deficiencies and provide counterfactual explanations of optimal actions. To achieve such results, CGPO proposes a bilevel mixed-integer nonlinear optimization framework for optimizing policies in defined expressivity classes (e.g. piecewise linear) and reduces it to an optimal constraint generation methodology that adversarially generates worst-case state trajectories. Furthermore, leveraging modern nonlinear optimizers, CGPO can obtain solutions with bounded optimality gap guarantees. We handle stochastic transitions through chance constraints, providing high-probability performance guarantees. We also present a roadmap for understanding the computational complexities of different expressivity classes of policy, reward, and transition dynamics. We experimentally demonstrate the applicability of CGPO across various domains, including inventory control, management of a water reservoir system, and physics control. In summary, CGPO provides structured, compact and explainable policies with bounded performance guarantees, enabling worst-case scenario generation and counterfactual policy diagnostics.
\end{abstract}


\section{Introduction}

An important aim of sequential decision optimization of challenging \emph{Discrete-Continuous Markov Decision Processes} (DC-MDP) in the artificial intelligence, operations research, and control domains is to derive policies that achieve optimal control.
A desirable property of such policies is \emph{compactness} of representation, which provides efficient execution on resource-constrained systems such as mobile devices \citep{wang2022}, as well as the potential for \emph{introspection and explanation} \citep{topin2021}. Moreover, while the derived policy is expected to perform well in expectation, in many applications it is desirable to obtain \emph{bounds on maximum policy error} and the scenarios that induce worst-case policy performance~\citep{corso2021}.

Popular policy optimization approaches used in model-free reinforcement learning~\citep{ppo2017} do not provide bounded error guarantees on policy performance, even if they use compact policy classes such as decision trees~\citep{topin2021}.  Model-based extensions that leverage gradient-based policy optimization~\citep{bueno2019} similarly do not provide bounded error guarantees due to the nonconvexity of the problem.  

In contrast, there exists a rich tradition of work leveraging \emph{mixed integer programming} (MIP) for bounded-error policy optimization~\citep{mmr17mdp,dolgov05cmdp,vos2023}, but {these methods only apply to discrete state and action MDPs}.  More importantly, these previous policy optimization formulations cannot be directly extended to continuous state or action MDPs due to the infinite number of constraints arising from the infinite state and action space in these formulations.  While some historical work has attempted to circumvent large or infinite constraint sets (from the continuous setting) via constraint generation~\citep{farias2006random,hauskrecht06hybridALP,schuurmans01cgenALP}, none of these works optimize for policies.
Finally, while there exist bespoke bounded error policy solutions for niche MDP classes with continuous states or actions such as linear–quadratic controllers \citep{aastrom2012}, ambulance dispatching \citep{albert2022}, and the celebrated ``$(s, S)$'' threshold policies for \citep{scarf1960}'s inventory management, these solution methods do not readily generalize beyond the specialized domains for which they are designed.


{\bf It remains an open question as to how to derive bounded error policies for the infinite state and action, discrete and continuous MDP (DC-MDP) setting.  We address this open question through a novel bilevel MIP formulation.}  Unfortunately, unlike MIP solutions for discrete MDPs, DC-MDPs with continuous states and/or actions cannot be solved directly since the resulting MIP  (a)~requires a challenging {bilevel MIP} formulation, (b)~has an {uncountably infinite} number of constraints, and (c)~will often be {nonlinear}.  
Fortuitously, we show how a constraint generation form of this policy optimization -- termed CGPO -- is able to address (a), (b) and (c) to optimize policies with bounded error guarantees in {many cases}. 

\begin{figure*}
    \centering
   \begin{minipage}{0.7\textwidth}
      \includegraphics[width=0.98\linewidth]{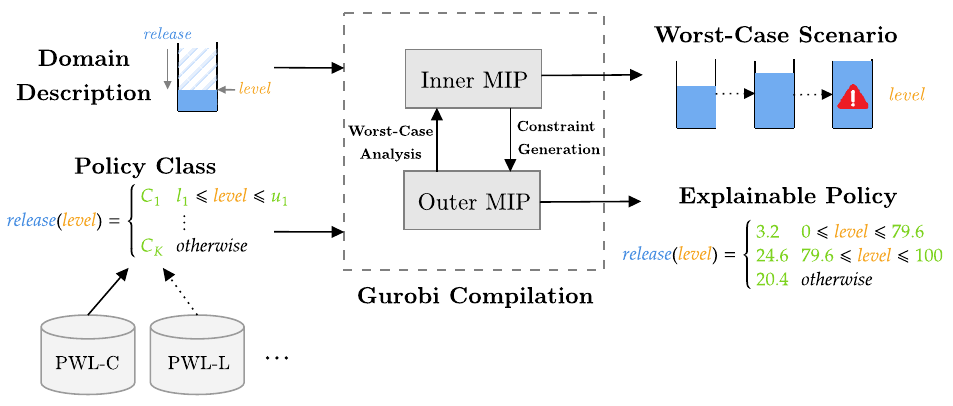}
    \end{minipage}\hfill%
    \begin{minipage}{0.3\textwidth}
      \includegraphics[width=0.95\linewidth]{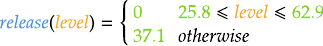}\\ \\
      \includegraphics[width=0.95\linewidth]{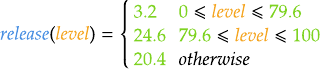}\\ \\
      \includegraphics[width=1.0\linewidth]{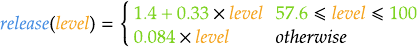}
    \end{minipage}
    \caption{Reservoir control is used as an illustrative example. \textbf{Left:} an overview of CGPO, which consists of a domain description and policy representation compiled to a bilevel mixed-integer program (MIP), in which the inner problem computes the worst-case trajectories for the current policy while the outer problem updates the policy via constraint-generation. The result is a worst-case scenario for the policy (facilitating policy failure analysis), a concrete policy within the expressivity class (for direct policy inspection), and a gap on its performance (error bound). \textbf{Right:} three optimal (i.e. zero-gap) policies produced upon termination across several piecewise policy classes. Crucially, our framework provides the ability to derive highly {compact} (e.g. memory and time-efficient to execute), {intuitive} and {nonlinear} policies, with {strong bound guarantees} on policy performance.}  
    \label{fig:enter-label}
\end{figure*}

In summary, our aim is to provide a solution approach to optimize and provide strong performance bound guarantees on structured and compact DC-MDP policies under various expressivity classes of policy, reward, and nonlinear transition dynamics. Our specific contributions are:
\begin{enumerate}
    \item \emph{We propose a novel bilevel optimization framework for solving nonlinear DC-MDPs called Constraint-Generation Policy Optimization (CGPO), that admits a clever reduction to an iterative constraint-generation algorithm that can be readily implemented using standard MIP solvers \mbox{(Fig.~\ref{fig:enter-label})}.}
    State-of-the-art MIP solvers often leverage spatial branch-and-bound techniques, which can provide not only optimal solutions for large mixed integer linear programs (MILP), but also 
    bounded optimality guarantees for mixed integer {nonlinear} programs (MINLP) \citep{castro2015tightening}.  Chance constraints~\citep{farina2016} are used for probabilistic guarantees. 
    \item \emph{If the constraint generation algorithm terminates, we guarantee that we have found the optimal policy within the given policy expressivity class (Theorem \ref{thm:main}).} Further, even when constraint generation does not terminate, our algorithm provides a tight optimality bound on the performance of the computed policy at each iteration and the scenario where the policy performs worst (and as a corollary, for any externally provided policy).
    \item \emph{We provide a road map to characterize the optimization problems in (1) above -- and their associated expressivity classes -- for different expressivity classes of policies and state dynamics, ranging from MILP, to polynomial programming, to  nonlinear programs (Table~\ref{table:complexity}).} This information is beneficial for reasoning about which optimization techniques are most effective for different combinations of DC-MDP and policy class expressivity.
    \item Finally, \emph{we provide a variety of experiments demonstrating the range of rich applications of CGPO under linear and nonlinear dynamics and various policy classes.} Critically, we derive bounded optimal solutions for a range of problems, including linear inventory management, reservoir control, and a highly nonlinear VTOL control problem. Furthermore, since our policy classes are compact by design, we can also directly inspect and analyze these policies (Fig.~\ref{fig:policies}). To facilitate policy interpretation and diagnostics, we can compute the state and exogenous noise trajectory scenario that attains the worst-case error bounds for a policy (Fig.~\ref{fig:worst})
    as well as a {counterfactual explanation} of what action {should} have been taken in comparison to what the policy prescribes.
\end{enumerate}

\section{Preliminaries}
\label{sec:prelim}

\subsection{Function Classes}

We first define some important classes of functions commonly referred to throughout the paper. Let $\mathcal{X}$ and $\mathcal{Y}$ be arbitrary sets. A function $f : \mathcal{X} \to \mathcal{Y}$ is called \emph{piecewise} (PW) if there exist functions $f_1, \dots f_K, f_{K+1} : \mathcal{X} \to \mathcal{Y}$ and disjoint Boolean case conditions $\mathcal{P}_1, \dots \mathcal{P}_K : \mathcal{X} \to \lbrace 0, 1 \rbrace$, such that for any $x \in \mathcal{X}$, $\forall i, j\!\neq\!i$ we have $\mathcal{P}_i(x)\!=\!1 \, \Rightarrow \, \mathcal{P}_j(x)\!=\!0$, and for case functions $f(x)=f_i(x)$ if $\mathcal{P}_i(x)=1$ and $f(x)=f_{K+1}(x)$ if $\mathcal{P}_i(x) = 0, \, \forall i$.

Both case conditions $\mathcal{P}_i$ and case functions ${f}_i$ can be linear or nonlinear. In this paper, we consider the following classes:
\begin{description}
    \item[C] \emph{constant}, i.e. $f_i(\mathbf{x}) = C$
    \item[D] \emph{discrete}, same as restricted values of C
    \item[S] \emph{simple (axis-aligned) linear functions}, i.e. $f_i(\mathbf{x}) = b_i + w_i \times x_j$, where $j \in \lbrace 1, 2 \dots n \rbrace$ and $b_i, w_i \in \mathbb{R}$
    \item[L] \emph{linear}, i.e. $f_i(\mathbf{x}) = b_i + \mathbf{w}_i^T \mathbf{x}$
    \item[B] \emph{bilinear}, i.e. $f_i(\mathbf{x}) = x_1 \times x_2 + x_2 \times x_3$
    \item[Q] \emph{quadratic}, i.e. $f_i(\mathbf{x}) = b_i + \mathbf{w}_i^T \mathbf{x} + \mathbf{x}^T M_i \mathbf{x}$, where $M_i \in \mathbb{R}^{n\times n}$
    \item[P] \emph{polynomial} of order $m$, i.e. $f_i(\mathbf{x}) = \sum_{j_1=1}^m \dots \sum_{j_n = 1}^m w_{i,j_1,\dots j_n} \prod_{k=1}^n x_k^{j_k}$
    \item[N] \emph{general nonlinear}, i.e. $f(\mathbf{x}) = e^{b_i + \mathbf{w}_i^T \mathbf{x}}$.
\end{description}
Analogously, we consider analogous functions over integer (I) domains $\mathcal{X}$ or $\mathcal{Y}$, or mixed discrete-continuous (M) domains. $\mathcal{P}_i$ can be characterized similarly based on whether the constraint set defining it is constant, linear, bilinear, etc. 

We also introduce a convenient notation to describe general PW functions by concatenating the expressivity classes of their constraints $\mathcal{P}_i$ and case functions $f_i$. For example, PWS-C describes a piecewise function on $\mathbb{R}^n$ with simple (axis-aligned) linear constraints and constant values, while PWS-L describes a function with similar constraints but linear values.
We can also append the number of cases $K$ to the notation, i.e. PWS1-L describes a piecewise linear function with $K = 1$ case conditions. 
More generally, for PWL, PWP or PWN, constraints could also be logical conjunctions of other simpler constraints.

\subsection{Mathematical Programming}

Given a function $f : \mathcal{X} \to \mathcal{Y}$ and Boolean case condition $\mathcal{P} : \mathcal{X} \to \lbrace 0, 1 \rbrace$, the \emph{mathematical program} is defined as:
\begin{equation*}
\begin{aligned}
    \min_{x \in \mathcal{X}}\quad f(x) \qquad 
    \textrm{s.t.} \quad \mathcal{P}(x) = 1.
\end{aligned}
\end{equation*}

Important expressivity classes of \emph{mixed-integer program} (MIP) optimization problems, i.e. those with discrete and continuous decision variables, include:
\begin{description}
    \item[MILP] \emph{mixed-integer linear}: $f$ and $\mathcal{P}$ both in L
    \item[MIBCP] \emph{mixed-integer bilinearly constrained}: $f$ is in L and $\mathcal{P}$ is in B
    \item[MIQP] \emph{mixed-integer quadratic}: $f$ is in Q and $\mathcal{P}$ is in L
    \item[QCQP] \emph{quadratically constrained quadratic}: $f$ and $\mathcal{P}$ both in Q
    \item[PP] \emph{polynomial}: $f$ and $\mathcal{P}$ both in P
    \item[MINLP] \emph{mixed-integer nonlinear}: $f$ and $\mathcal{P}$ both in N.
\end{description}


Branch-and-bound \citep{morrison2016branch} solvers
are commonly used to maintain upper and lower bounds on the minimal objective value of any linear or nonlinear MIP,
whose difference is the so-called \emph{optimality gap}; when this gap is zero, an optimal solution has been found. Moreover, some packages such as Gurobi, which we use to perform the necessary compilations in CGPO in our experiments, also support a variety of nonlinear mathematical operations via piecewise-linear approximation \citep{castro2015tightening}.

\subsection{Discrete-Continuous Markov Decision Processes}

A \emph{Discrete-Continuous Markov decision process} (DC-MDP) is a tuple $\langle \states, \actions, P, r \rangle$, where $\states$ is a set of states, $\actions$ is a set of actions or controls.  $\states$ and $\actions$ may be discrete, continuous, or mixed.  $P(\st{}, \at{}, \st{}')$ is the probability of transitioning to state $\st{}'$ immediately upon choosing action $\at{}$ in state $\st{}$, and $r(\st{}, \at{}, \st{}')$ is the corresponding reward received. We assume that $r$ is uniformly bounded, i.e. there exists $B < \infty$ such that $|r(\st{}, \at{}, \st{}')| \leq B$ holds for all $\st{}, \at{}, \st{}'$.

Given a planning horizon of length $T$ (assumed fixed in our setting), the \emph{value} of a (open-loop) \emph{plan} $\alpha = [\at{1}, \dots \at{T}] \in \mathcal{A}^T$ starting in state $\st{1}$ is 
\begin{equation*}
    V(\alpha, \st{1}) = \E{\st{t+1} \sim P(\st{t}, \at{t}, \cdot)}{\sum_{t=1}^T r(\st{t}, \at{t}, \st{t+1})}.
\end{equation*}
A (closed-loop) \emph{policy} $\pi = [\pi_1, \dots \pi_T]$ is a sequence of mappings $\pi_t : \states \to \actions$, whose value is defined as
\begin{equation*}
    V(\pi, \st{1}) = \E{\st{t+1} \sim P(\st{t}, \pi_t(\st{t}), \cdot)}{\sum_{t=1}^T r(\st{t}, \pi_t(\st{t}), \st{t+1})}.
\end{equation*}
{Dynamic programming} approaches such as value and policy iteration can compute an optimal horizon-dependent policy $\pi_{H}^*$ \citep{puterman2014}, but do not directly apply to DC-MDPs with infinite or continuous state or action spaces. 

Our goal is to compute an optimal stationary policy $\pi^*$ that minimizes the error in value relative to $\pi_H^*$ over all initial states of interest $\states_1 \subseteq \states$, and all stationary policies $\Pi$, i.e.
\begin{equation}
\label{eqn:optimality}
    \pi^* \in \argmin_{\pi \in \Pi} \max_{\st{1} \in \states_1} \left[ V(\pi_H^*, \st{1}) - V(\pi, \st{1}) \right].
\end{equation}

In practical applications, the policy class is often restricted to function approximations $\tilde{\Pi} \subset \Pi$. A variety of planning approaches can compute 
plans $\tilde{\pi}^* \in \tilde{\Pi}$ for this problem, including straight-line planning that scales well in practice but does not learn policies \citep{raghavan2017,wu2017}.  Reactive policy optimization~\citep{bueno2019,low2022} and variants of MCTS~\citep{swiechowski2023} that learn neural network policies cannot provide concrete error bounds on policy performance nor do they facilitate worst-case analysis, or ease of interpretation of learned compact policy classes as we provide.

\section{Methodology}
\label{sec:main}

We begin with a derivation of CGPO for general DC-MDPs in the deterministic setting. We then derive CGPO in the more general stochastic setting using chance constrained optimization.  We end with a discussion of problem class complexity and convergence. A worked example illustrating the overall execution of CGPO is provided in Appendix~\ref{sec:app-worked-example}.


\subsection{Constraint Generation for Deterministic DC-MDPs}

We assume $\pi$ can be compactly identified by a vector $\mathbf{w} \in \mathcal{W}$ of decision variables, and we use the shorthand $V(\mathbf{w}, \st{1})$ to denote the value $V(\pi_{\mathbf{w}}, \st{1})$ of policy $\pi_{\mathbf{w}}$ parameterized by $\mathbf{w}$. We focus our attention on approximate policy sets $\tilde{\Pi}_{\mathcal{W}} = \lbrace \pi_{\mathbf{w}} : \mathbf{w} \in \mathcal{W} \rbrace$ enumerated by a compact set of parameters $\mathcal{W}$.

First, observe that for every possible initial state $\st{1}$, there exists a fixed optimal plan $\alpha^* = [\at{1}^*, \dots \at{T}^*]$ such that $V(\alpha^*, \st{1}) = V(\pi_H^*,\st{1})$. On the other hand, since we only have access to an expressivity class of {approximate} policies $\tilde{\Pi}_{\mathcal{W}}$, the error $\varepsilon(\mathbf{w}, \st{1})$ of $\pi_{\mathbf{w}}$ in state $\st{1}$ relative to $V(\pi_H^*,\st{1})$ according to (\ref{eqn:optimality}) must be
\begin{equation*}
    \varepsilon(\mathbf{w}, \st{1}) \geq \max_{\alpha \in \actions^T} V(\alpha, \st{1}) - V(\mathbf{w}, \st{1}),
\end{equation*}
and thus the (global) {worst-case} error $\varepsilon(\mathbf{w})$ of $\mathbf{w}$ is
\begin{equation}
\label{eqn:pi_lower_bound}
    \varepsilon(\mathbf{w}) \geq \max_{\st{1} \in \states_1} \max_{\alpha \in \actions^T} \left[ V(\alpha, \st{1}) - V(\mathbf{w}, \st{1}) \right].
\end{equation}
However, since we seek the approximate {optimal} policy $\tilde{\pi}^* \in \tilde{\Pi}_{\mathcal{W}}$, we can directly minimize (\ref{eqn:pi_lower_bound}) over $\mathcal{W}$, obtaining the \emph{infinitely-constrained mixed-integer program}:
\begin{equation} 
\label{eqn:bilevel}
    \begin{aligned}
        \min_{\mathbf{w} \in \mathcal{W}, \, \varepsilon \in [0, \infty)} \quad & \varepsilon \\
        \textrm{s.t.} \quad & \varepsilon \geq \max_{\st{1}\in\states_1}  \max_{\alpha \in \actions^T} \big[ V(\alpha, \st{1})- V(\mathbf{w}, \st{1}) \big].
    \end{aligned}
\end{equation}
However, the constraint is highly nonlinear, turning \eqref{eqn:bilevel} into a {bilevel program} and making analysis of this particular formulation difficult. Instead, since the $\max$'s must hold for all states and actions, we can rewrite the problem (\ref{eqn:bilevel}) as:
\begin{equation}
\label{eqn:bilevel2}
\begin{aligned}
    \min_{\mathbf{w} \in \mathcal{W}, \, \varepsilon \in [0, \infty)} \quad &\varepsilon  \\
    \textrm{s.t.} \quad &\varepsilon \geq V(\alpha, \st{1})- V(\pi, \st{1}), \quad \, \forall \alpha \in \actions^T, \, \st{1} \in \states_1.
\end{aligned}
\end{equation}
The goal is therefore to solve (\ref{eqn:bilevel2}), which is still infinitely-constrained when $\mathcal{S}$ or $\mathcal{A}$ are infinite spaces. Instead, we solve it by splitting the $\min$ over $\mathcal{W}$ and the $\max$ over $(\alpha, \st{1})$ into two problems, and apply \emph{constraint generation} \citep{blankenship1976infinitely,chembu2023}. 

Specifically, starting with a fixed arbitrary scenario $(\st{1}, \alpha)$, we form the constraint set $\mathcal{C} = \lbrace (\st{1}, \alpha) \rbrace$, and then solve the following two problems:
\begin{description}
    \item[\textbf{Outer}] Solve (\ref{eqn:bilevel2}) within the set of finite constraints $\mathcal{C}$ to obtain $\mathbf{w}^* \in \mathcal{W}$.
    \item[\textbf{Inner}] Solve (\ref{eqn:pi_lower_bound})
    ${\arg\!\max}_{\st{1} \in \states_1, \alpha \in \actions^T} \!\left[ V(\alpha, \st{1})\!-\!V(\mathbf{w}^*, \st{1}) \right]$, for the {highest-error scenario} for $\mathbf{w}^*$, and append it to $\mathcal{C}$.
\end{description}
These two steps are repeated until the constraint added to the outer problem is no longer binding, i.e. $V(\alpha^*, \st{1}^*) - V(\mathbf{w}^*, \st{1}^*) \leq 0$, since the solution of the outer problem will not change with the addition of the new constraint. We name our approach \emph{Constraint-Generation Policy Optimization} (CGPO). \emph{In retrospect, we can view CGPO as a policy iteration algorithm where the inner problem adversarially critiques the policy with a worst-case trajectory and the outer problem improves the policy w.r.t.\ all critiques}.

\paragraph{Remarks} Upon termination, CGPO guarantees an optimal (lowest error) policy within the specified policy class (Theorem \ref{thm:main}). 
While we cannot provide a general finite-time guarantee of termination with an optimal policy, CGPO is an {anytime} algorithm that provides a best policy and a bound on performance at each iteration. 
At each iteration, the solution to the inner problem allows analysis of the worst-case trajectory 
with respect to $\pi$ (Fig.~\ref{fig:worst}), and generates a counterfactual explanation of what actions {should} have been made.

\subsection{Chance Constraints for Stochastic DC-MDPs} 



\emph{Chance-constrained optimization} \citep{ariu2017,farina2016,ono2015} allows us to derive high probability intervals on stochastic MDP transitions and reduce our solution to a robust optimization problem with probabilistic performance guarantees.
To achieve this, we will require that the state transition function $P(\st{}, \at{}, \st{}')$ has a natural reparameterization as a deterministic function $g$ of $(\st{}$, $\at{})$ and some exogenous independent and identically distributed noise variable $\xi$ with density $q(\cdot)$ on support $\Xi$, e.g. $\st{}' = g(\st{}, \at{}, \xi)$ \citep{bueno2019,patton2022}. Given a threshold $p \in (0, 1)$ close to 1, we also assume the existence of a computable interval $\Xi_p = [\xi_l, \xi_u]$ such that $\prob{\xi_l \leq \xi \leq \xi_u} \geq p$.

Thus, we can repeat the derivation of (\ref{eqn:bilevel}) by considering the worst case not only over $\st{1}$, but also over possible noise variables $\xi_{1:T} = [\xi_1, \dots \xi_T] \in \Xi_p^T$. The final problem is:
\begin{equation}
\label{eqn:bilevel-cc}
\begin{aligned}
    \min_{\mathbf{w} \in \mathcal{W}, \, \varepsilon \in [0, \infty)} \quad & \varepsilon  \\
    \textrm{s.t.}\quad &\varepsilon \!\geq\! \max_{\st{1}\in\states_1} \max_{\xi_{1:T} \in \Xi_p^T} \max_{\alpha \in \actions^T} \left[ V(\alpha, \st{1}, \xi_{1:T})\!-\! V(\mathbf{w}, \st{1}, \xi_{1:T}) \right],
\end{aligned}  
\end{equation}
where $V(\cdot, \st{1}, \xi_{1:T})$ corresponds to the total reward of the policy or plan accumulated over trajectory $\xi_{1:T}$ starting in $\st{1}$. In the stochastic setting, 
$\varepsilon$ only holds with probability $p^T$ for a horizon $T$ problem. Experimentally, we found that choosing $p$ such that $p^T = P$, where $P$ is a desired probability bound on the full planning trajectory, was sufficient\footnote{Hence, we arbitrarily chose $p = 0.995$ for our empirical evaluations.}.

Once again, (\ref{eqn:bilevel-cc}) can be reformulated as:
\begin{equation}
\label{eqn:bilevel2-cc}
    \begin{aligned}
    \min_{\mathbf{w} \in \mathcal{W}, \, \varepsilon \in [0, \infty)} \quad &\varepsilon \\
    \textrm{s.t.} \quad &\varepsilon \geq V(\alpha, \st{1}, \xi_{1:T})- V(\mathbf{w}, \st{1}, \xi_{1:T}), \quad \, \forall \alpha \in \actions^T, \, \st{1} \in \states_1, \, \xi_{1:T} \in \Xi_p^T.
    \end{aligned}
\end{equation}
Therefore, we can again apply constraint generation to solve this problem in two stages, in which the inner optimization produces not only a worst-case initial state and action sequence, but also the disturbances $\xi_{1:T}^*$ that reproduce {all} future worst-case state realizations $\st{2}, \dots \st{T}$. The overall workflow of the computations is summarized as Algorithm~\ref{alg:stochastic}.

\begin{algorithm}[!h]
    \caption{Constraint-Generation Policy Optimization (CGPO)}
    \label{alg:stochastic}
    \begin{algorithmic}
        \STATE Initialize $p \in (0, 1)$, $\Xi_p$, $\st{1} \in\mathcal{S}_1$, $\xi_{1:T} \in \Xi_p^T$, $\alpha \in \mathcal{A}^T$
        \STATE Set $t = 0$ and $\mathcal{C}_0 = \lbrace (\st{1}, \xi_{1:T}, \alpha) \rbrace$
                \STATE \textbf{Step 1:} Solve the \textbf{outer problem}:
                \begin{equation*}     
                \begin{aligned}
                    (\mathbf{w}_t^*, \varepsilon_t^*) \in \,{\arg\!\min}_{\mathbf{w}\in\mathcal{W},\,\varepsilon \in [0, \infty)} \, &\varepsilon \\
                    \mathrm{s.t.} \quad &\varepsilon \geq V(\alpha, \st{1}, \xi_{1:T}) - V(\mathbf{w}, \st{1}), \quad \, \forall  \,( \st{1}, \xi_{1:T}, \alpha) \in \mathcal{C}_t
                \end{aligned}
                \end{equation*}   
                
                \STATE \textbf{Step 2:} Solve the \textbf{inner problem}:
                \begin{equation*}
                    \begin{aligned}
                    &(\st{1}^*, \xi_{1:T}^*, \alpha^*) \in {\arg\!\max}_{\st{1} \in \states_1, \, \xi_{1:T} \in \Xi_p^T,\, \alpha \in \actions^T} \left[                    V(\alpha, \st{1}, \xi_{1:T})- V(\mathbf{w}_{t}^*, \st{1}, \xi_{1:T}) \right]
                    \end{aligned}    
                \end{equation*}                
                \textbf{Step 3:} Check convergence:
                \IF{$V(\mathbf{w}_{t}^*, \st{1}^*, \xi_{1:T}^*) \geq V(\alpha^*, \st{1}^*, \xi_{1:T}^*)$ (or $\|\mathbf{w}_t^* - \mathbf{w}_{t-1}^*\| < \delta$)}
                    \STATE Terminate with policy $\pi_{\mathbf{w}_{t}^*}$ and error $\varepsilon_t^*$ 
                \ELSE
                    \STATE Set $t = t + 1$, $\mathcal{C}_{t+1} = \mathcal{C}_{t} \cup \lbrace (\st{1}^*, \xi_{1:T}^*, \alpha^*) \rbrace$ and \textbf{go to} \textbf{Step 1}
                \ENDIF
    \end{algorithmic}
\end{algorithm}

\paragraph{Remark} Algorithm \ref{alg:stochastic} may not terminate in general, unless convexity assumptions are placed on $V(\alpha, \st{})$ or $V(\mathbf{w}, \st{})$. Thus, in typical applications, one would terminate when $\|\mathbf{w}_t^* - \mathbf{w}_{t-1}^*\| < \delta$ for chosen hyper-parameter $\delta$. However, our experiments have shown empirically that Algorithm \ref{alg:stochastic} terminates with an exact optimal solution for most policy classes considered, at which point both the return and error curves ``plateau'' (cf. Fig. \ref{fig:simulation}).


\subsection{Convergence}

Under mild conditions, if CGPO terminates at some iteration $t$, then $\mathbf{w}_t^*$ is optimal in $\mathcal{W}$ (with probability at least $p^T$ in the stochastic case).

\begin{theorem}
\label{thm:main}
    If $\mathcal{S}_1$, $\mathcal{A}$, $\Xi_p$ and $\mathcal{W}$ are non-empty compact subsets of Euclidean space, $V(\alpha, \st{}, \xi_{1:T})$ and $V(\mathbf{w}, \st{}, \xi_{1:T})$ are continuous, and CGPO terminates at iteration $t$, then $\mathbf{w}_t^*$ is optimal for problem (\ref{eqn:bilevel2-cc}).
\end{theorem}

\begin{proof}
    In the notation of \citep{blankenship1976infinitely}, we define $X^0 = [0, 2 B T]$, $X = \mathcal{W}$, $Y = \mathcal{A}^T \times \mathcal{S}_1 \times \Xi_p^T$. Making the change of variable $\varepsilon = x_0 \in X^0$, $x = \mathbf{w}$, and $y = (\alpha, \st{1}, \xi_{1:T}) \in Y$, and defining the continuous function $f(x, y) = \varepsilon(\alpha, \st{1}, \xi_{1:T}) = V(\alpha, \st{1}, \xi_{1:T}) - V(\mathbf{w}, \st{1}, \xi_{1:T})$, the problem (\ref{eqn:bilevel2-cc}) is equivalent to the problem:
\begin{equation*}
\begin{aligned}
    \min_{(x_0, x) \in X^0 \times X} \, x_0 \qquad
    \mathrm{s.t.} \quad f(x, y) - x_0 \leq 0, \quad \forall y \in Y
\end{aligned}
\end{equation*}
as discussed on p. 262 of \citep{blankenship1976infinitely}. Similarly, Algorithm \ref{alg:stochastic} can be reparameterized as Algorithm 2.1 in the aforementioned paper, with $x_t = \mathbf{w}_t^*$. Thus, the assumption of Theorem 2.1 in \citep{blankenship1976infinitely} holds and $\mathbf{w}_t^*$ is a solution of (\ref{eqn:bilevel2-cc}) as claimed. 
\end{proof}

\paragraph{Remark} Termination of CGPO guarantees that $\mathbf{w}_t^*$ is optimal in $\tilde{\Pi}_\mathcal{W}$, and $\varepsilon_t^*$ is the corresponding optimality gap. Moreover, if $\varepsilon_t^* = 0$ and the MDP is deterministic, then $\pi_{\mathbf{w}_t^*}$ is the {optimal} policy in $\Pi$. This suggests the gap estimate can be used as a principled way to assess the suitability of different policy classes. 


\subsection{Problem Expressivity Class Analysis}

\begin{table}
\centering
 \begin{tabular}{|c || c | c | c |} 
 \cline{2-4}
 \multicolumn{1}{c|}{} & \multicolumn{3}{|c|}{Dynamics / Reward} \\  \hline
 Policy & L & P & N \\
 \hline
 PWS-\{C, D\} & MILP & PP & MINLP \\
 PWL-\{C, D\} & MILP/MIBCP & PP & MINLP \\
 S, L & MILP/PP & PP & MINLP \\
 PW\{S,L\}-\{S,L\} & MILP/PP & PP & MINLP \\ 
 PW\{S,L,P\}-P & PP & PP & MINLP \\
 Q & PP & PP & MINLP \\
 PWN-N & MINLP & MINLP & MINLP  \\
 \hline
 \end{tabular} 
 \caption{Problem classes of the inner/outer optimization problems in CGPO for different expressivity classes of dynamics/reward (columns) and policies (rows).}
\label{table:complexity}
\end{table}

In Table \ref{table:complexity}, we present the relationship between the expressivity classes of policies and state transition dynamics and the corresponding classes of the inner and outer optimization problems. The policy classes of interest include PWS-C, PWL-C, PWS-L, PWL-L, PWN-N and the different variants of piecewise polynomial policies. Meanwhile, expressivity classes for state dynamics and reward include linear, polynomial and general nonlinear functions (their piecewise counterparts generally fall under the same expressivity classes, and are excluded for brevity). Interestingly, as shown in Appendix~\ref{sec:app-linear-case},
when the policy and state dynamics are both linear, the outer problem is a PP. In a similar vein, a PWL-C policy and linear dynamics result in a MIBCP outer problem due to the bilinear interaction between successor state decision variables and policy weights in the linear conditions. 

Our experiments in the next section empirically evaluate PWS-C and PWS-L policies with linear dynamics and Q policies with nonlinear dynamics. This requires solving mixed-integer problems with large numbers of decision variables ranging from MILP to PP to MINLP. 

\section{Empirical Evaluation}
\label{sec:eval}

We empirically validate the performance of CGPO on several MDPs, aiming to answer the following questions:

\begin{enumerate}[leftmargin=*, label={\textbf{Q\arabic*}}]
    \item Does CGPO recover exact solutions when the ideal policy class is known? How does it perform if the optimal policy class is not known?
    \item How do different policy expressivity classes perform for each problem?
    \item Does the worst-case analysis provide further insight about a policy?  
\end{enumerate}

\subsection{Domain Descriptions}

To answer these questions, we evaluate on linear Inventory, linear Reservoir, and nonlinear VTOL (vertical take-off and landing) domains as summarized below. Inventory has provably optimal PWS-S policies, whereas no optimal policy class is known explicitly for Reservoir.
Experimental details and additional results are discussed in Appendix~\ref{sec:app-details} and~\ref{sec:app-more-results} respectively, and a public GitHub repository\footnote{\scriptsize\url{https://github.com/pyrddlgym-project/pyRDDLGym/tree/GurobiCompilerBranch/}} allows reproduction of all results and application of CGPO to arbitrary domains. 

\paragraph{Linear Inventory} State $s_{t}$ describes the {discrete} level of stock available for a single good, action $a_{t} \in [0, B]$ is the {discrete} reorder quantity, and demand $\xi_{t}$ is stochastic and distributed as discrete uniform:
\begin{align*}
    s_{t+1} &= s_t + a_t - \lfloor \xi_t \rfloor, \quad \xi_{t} \sim \mathrm{Uniform}(L, U).
\end{align*}
We permit backlogging of inventories, represented as negative $s_t$. We define costs $C, P$ and $S$ which represent, respectively, the purchase cost, excess inventory cost and shortage cost, and the reward function can be written as
\begin{align*}
    r(\st{t+1}) &= -C \times a_t - P \times (s_{t+1})_+ - S  \times (-s_{t+1})_+,
\end{align*}
where $(\cdot)_+ = \max[0, \cdot]$. We set $B = 10$, $C=0.5$, $P=S=2$, $L=2$, $U=6$. If $P > C$ and $T = \infty$, then a PWS-S policy is provably optimal (otherwise if $T < \infty$, then it may be non-stationary). A planning horizon of $T = 8$ is used. For this domain and reservoir below, we focus on learning factorized piecewise policies, i.e. C, S, PWS-C and PWS-S. 

\paragraph{Linear Reservoir} The goal is to manage the water level in a system of interconnected reservoirs. State $s_{t,r}$ represents the {continuous} water level of reservoir $r$ with capacity $M_r$, action $a_{t,r} \in [0, s_{t,r}]$ is the {continuous} amount of water released, and rainfall $\xi_{t,r}$ is a truncated normally-distributed random variable. Each reservoir $r$ is connected to a set $U(r)$ of upstream reservoirs, thus:
\begin{align*}
    s_{t+1,r} = \min\Big[M_r, \Big(s_{t,r} + (\xi_{t,r})_+ - a_{t,r} + \sum_{d \in U(r)} a_{t,d}\Big)_+ \Big], \quad
    \xi_{t,r} \sim \mathcal{N}(m_r, v_r).
\end{align*}
Reward linearly penalizes any excess water level above $H_r$ and below $L_r$:
\begin{align*}
    r(\st{t+1}) &= \sum_{r=1}^N l_r (L_r - s_{t+1,r})_+ + h_r (s_{t+1,r} - H_r)_+.
\end{align*}
Our experiment uses a two-reservoir system with the following values for reservoir 1 $L_1=20, H_1=80, M_1=100, m_1=5, v_1=5, U(1)=\emptyset$, and the following values for reservoir 2 $L_2=30, H_2=180, M_2=200, m_2=10, v_2=10, U(2)=\lbrace 1 \rbrace$. Costs are identical for all reservoirs and are set to $l_r=-10, h_r=-100$. We use $T = 10$.

\paragraph{Nonlinear VTOL} 
The goal is to balance two masses on opposing ends of a long pole. The state consists of the angle $\theta_t$ and angular velocity $\omega_t$ of the pole, and the action is the force $F_t \in [0, 1]$ applied to the mass:
\begin{wrapfigure}{r}{0.32\textwidth}
    \vspace{-10pt}
  \begin{center}
    \includegraphics[width=\linewidth]{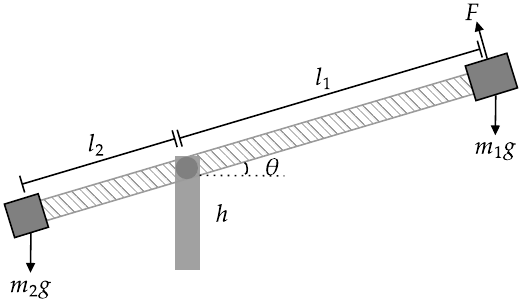}
  \end{center}
  \caption{Nonlinear VTOL system.}
    \vspace{-35pt}
  \label{fig:vtol}
\end{wrapfigure}
\begin{align*}
    \theta_{t+1} &= \max\left[-\sin(h/l_1), \min\left[\sin(h/l_2), \theta_t + \tau \omega_t \right] \right] \\
    \omega_{t+1} &= \omega_t + \frac{\tau}{J} \left(9.8 (m_2 l_2 - m_1 l_1) \cos \theta_t + 150 l_1 F_t \right) \\
    J &= m_1 l_1^2 + m_2 l_2^2 \\
    \textrm{s.t.} \quad & l_1 m_1 \geq l_2 m_2, \quad l_1 > l_2 > h, \quad 0 \leq F_t \leq 1.
\end{align*}

Time is discretized into intervals of $\tau = 0.1$ seconds, and we set $T = 6$. The reward penalizes the difference between the pole angle and a target angle:
\begin{equation*}
    r(F_t, \theta_{t+1}, \omega_{t+1}) = -|\theta_{t+1} - \theta_{target}|.
\end{equation*}
We use values $l_1=1, l_2=0.5, m_1=10, m_2=1, h=0.4, g=9.8$. We optimize for {nonlinear} quadratic policies.

\paragraph{Nonlinear Intercept} To demonstrate an example of {discrete} (Boolean) action policy optimization over nonlinear {continuous} state dynamics with independently moving but interacting projectiles, we turn to the Intercept problem \citep{scala2016}. Space restrictions require us to relegate details and results to Appendix~\ref{sec:app-intercept} and~\ref{sec:app-more-results}, but we remark here that CGPO is able to derive an optimal policy.

\paragraph{Initial State Bounds $\mathcal{S}_1$} For Reservoir, the initial state bounds are set to $[50, 100]$ for reservoir 1 and $[100, 200]$ for reservoir 2. For Inventory, the bounds are $[0, 2]$, and for VTOL and Intercept they are fixed to the initial state of the system.

\begin{figure*}[!htb]
    \centering
    \includegraphics[width=0.33\linewidth,trim={0.6cm 0 0.6cm 0},clip]{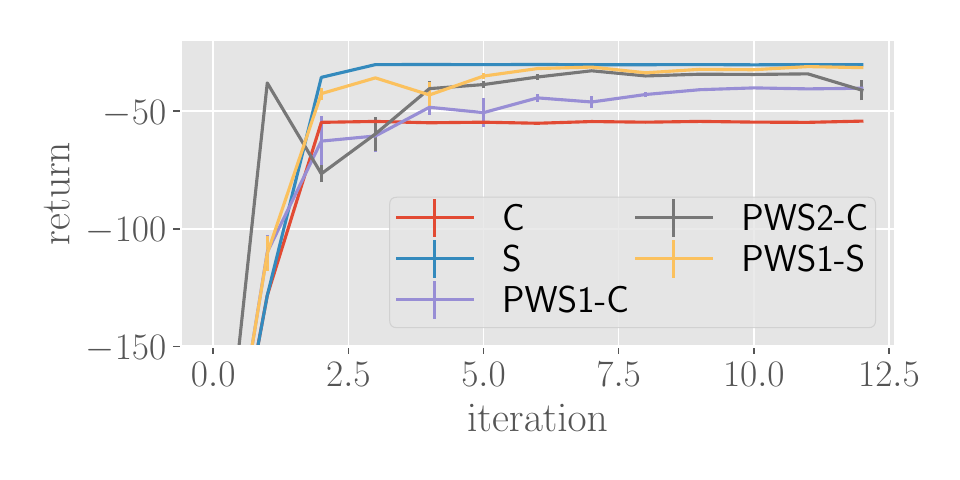}%
    \includegraphics[width=0.33\linewidth,trim={0.7cm 0 0.6cm 0},clip]{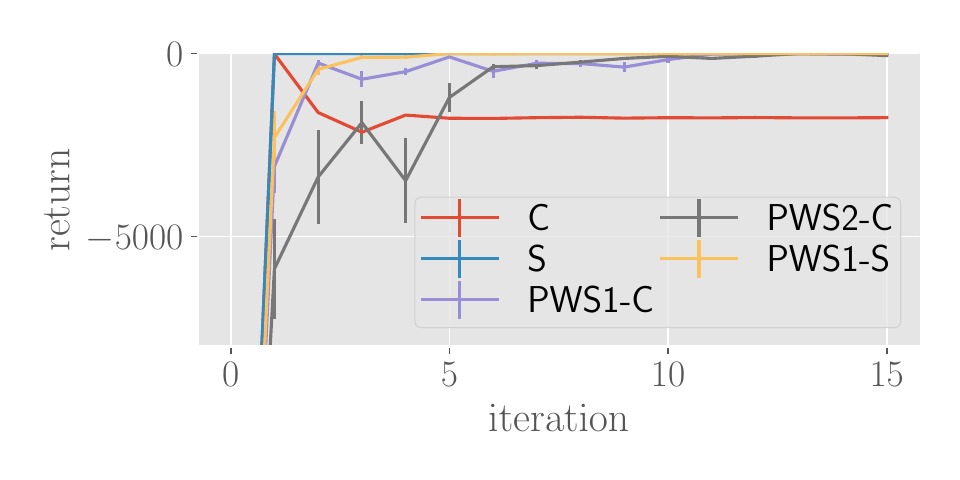}%
    \includegraphics[width=0.33\linewidth,trim={0.6cm 0 0.6cm 0},clip]{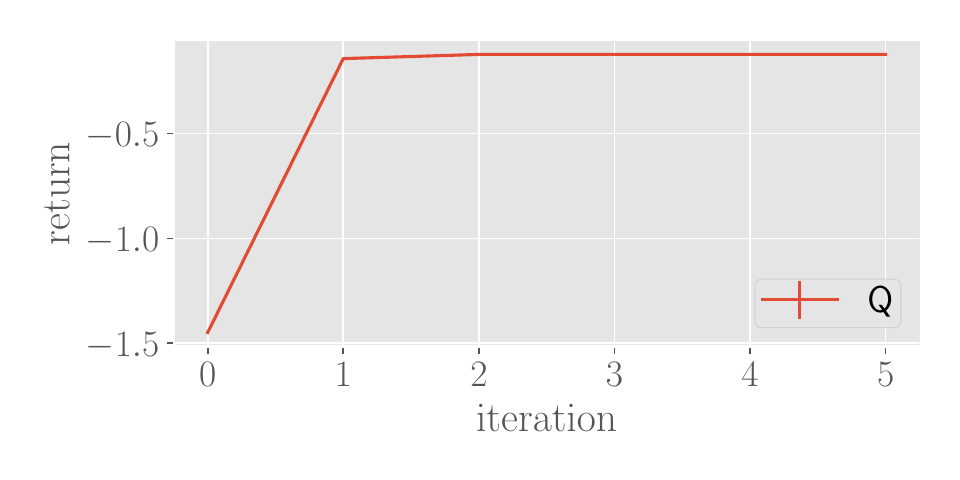}\\       
    \includegraphics[width=0.33\linewidth,trim={0.6cm 0 0.6cm 0},clip]{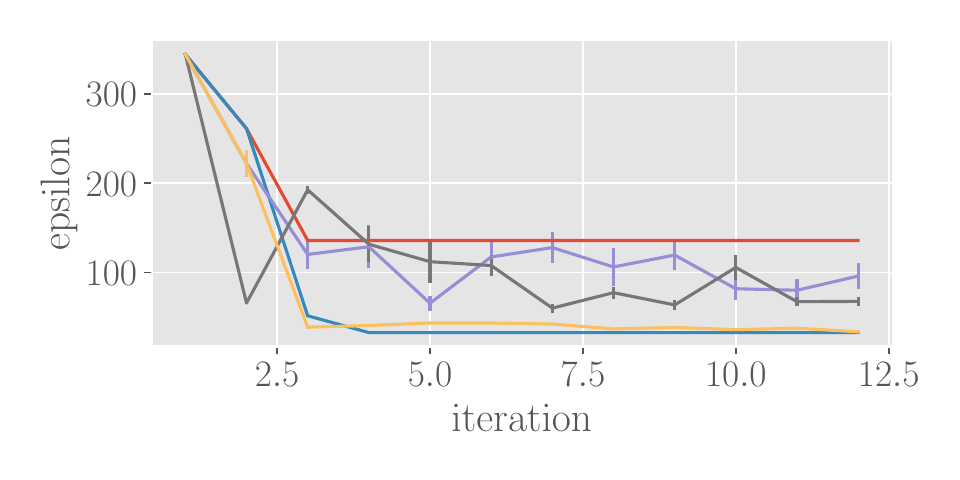}%
    \includegraphics[width=0.33\linewidth,trim={0.7cm 0 0.6cm 0},clip]{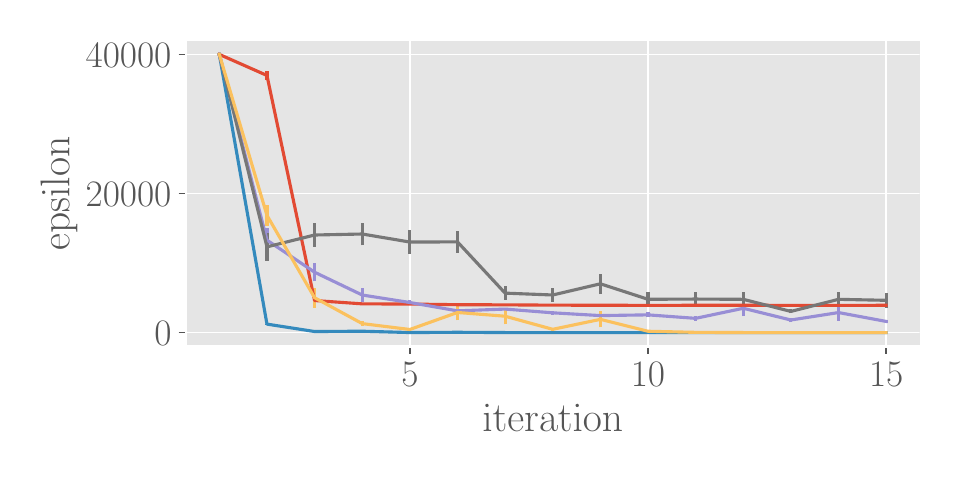}%
    \includegraphics[width=0.33\linewidth,trim={0.6cm 0 0.6cm 0},clip]{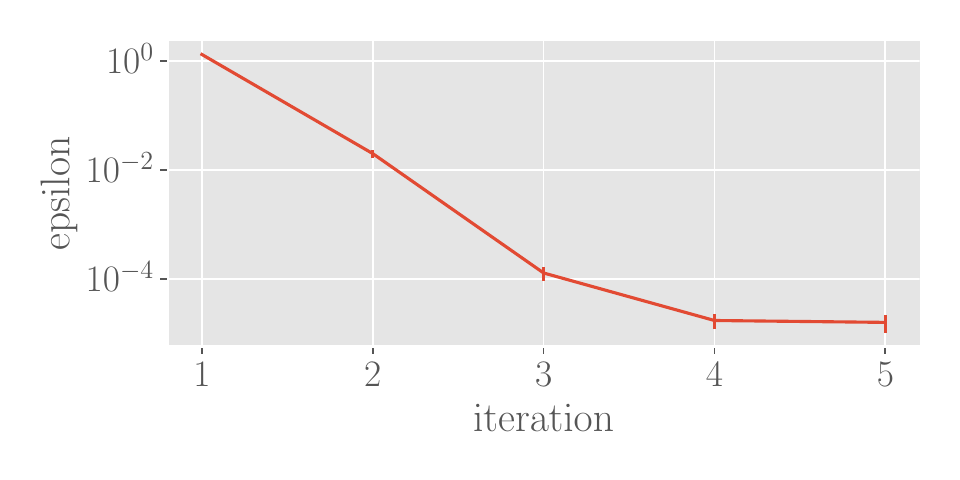}\\
    \caption{\textbf{Reservoir (left), Inventory (middle), VTOL (right):} Simulated return (top row) and worst-case error $\varepsilon^*$ (bottom row) over 100 roll-outs as a function of the number of iterations of constraint generation. Bars represent 95\% confidence intervals calculated using 10 independent runs of CGPO.}
    \label{fig:simulation}
\end{figure*}

\begin{figure*}[!htb]
    \centering
    \begin{tabular}{ccc}
        \resizebox{.33\hsize}{!}{%
            \begin{tabular}{@{}l@{}}
                \(
                    release_1 = \begin{cases}
                        33.23 &\mbox { if } 65.06 \leq level_1 \leq 100 \\
                        1.08  &\mbox{ otherwise }  
                    \end{cases}
                \) \\
                 \(
                    release_2 = \begin{cases}
                        72.91 &\mbox { if } 100 \leq level_2 \leq 200 \\
                        0.2 &\mbox{ otherwise } 
                    \end{cases}
                \)
            \end{tabular}
        } \hspace{0.2em} &
       \resizebox{.33\hsize}{!}{%
            \begin{tabular}{@{}l@{}}
                 \(
                    reorder = \begin{cases}
                        5 - stock &\mbox { if } -1 \leq stock \leq 2 \\
                        3 - stock &\mbox{ otherwise } 
                    \end{cases}
                \)
            \end{tabular}
        } &
        \resizebox{.3\hsize}{!}{%
             \begin{tabular}{@{}l@{}}                
                \includegraphics[width=0.5\linewidth]{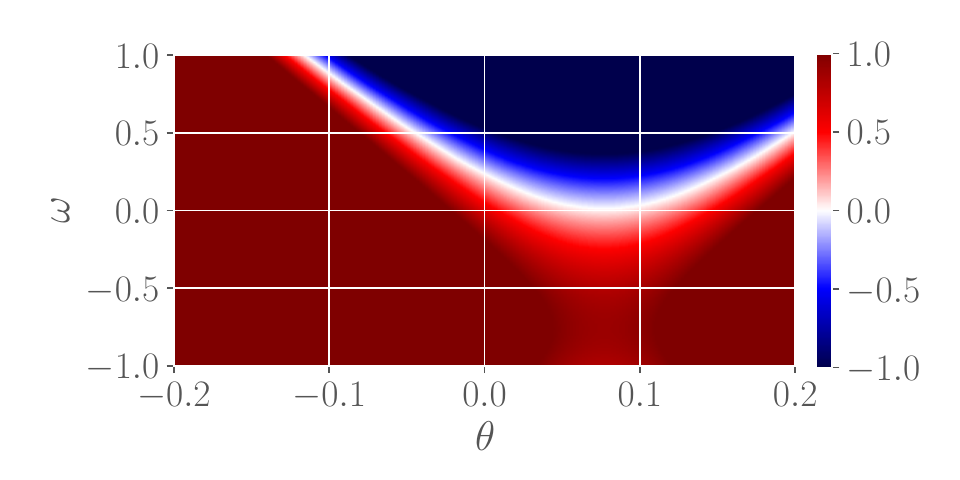}\\
                \(
                    F = 0.6 - 15.4 \theta - 2.3 \omega + 100 \theta^2 -1.5 \omega^2
                \)
            \end{tabular}}
    \end{tabular}
    \caption{\textbf{Reservoir (left), Inventory (middle), VTOL (right):} Examples of optimal policies computed at the end of CGPO.}
    \label{fig:policies}
\end{figure*}

\subsection{Empirical Results}

Fig. \ref{fig:simulation} compares the empirical (simulated) returns and error bounds $\varepsilon^*$ (optimal objective value of the inner problem) across different policy classes. 
Fig. \ref{fig:policies} illustrates examples of policies learned in a typical run of CGPO. As hypothesized, CGPO can recover exact $(s, S)$ control policies for Inventory in the PWS-S class, which together with S policies achieve the lowest error and best return across all policy classes. Interestingly, as Fig. \ref{fig:simulation} shows, C and PWS-C policies perform relatively poorly even for $K = 2$, but as expected, the error decreases as the expressivity of the policy class, and in particular the number of cases, increases. This is intuitive, as Fig. \ref{fig:policies} shows, the reorder quantity is strongly linearly dependent on the current stock, which is not easily represented as PWS-C. In constrast, on Reservoir the class of PWS-C policies perform comparatively well for $K \geq 1$, and achieve (near)-optimality despite requiring a larger number of iterations. As expected, policies typically release more water as the water level approaches the upper target bound. Finally, the optimal policy for VTOL applies a large upward force when the angle or angular velocity are negative, with relative importance being placed on angle. As Fig. \ref{fig:policies}~(right) shows, the force is generally greater when either the angle is below the target angle or the angular velocity is negative, and equilibrium is achieved by applying a modest upward force between the initial angle (0.1) and the target angle (0).

Fig. \ref{fig:other-policies-inventory} and \ref{fig:other-policies-reservoir} provide examples of policies learned in other expressivity classes for Inventory and Reservoir problems, respectively. In all cases, policies remain intuitive and easy to explain even as the complexity increases. For instance, PWS2-C policies for Inventory order more inventory as the stock level decreases, while for Reservoir, the amount of water released is increasing in the water level. 

\begin{figure*}[!ht]
    \centering
    \begin{tabular}{c c c c}
       \resizebox{.09\hsize}{!}{%
            \begin{tabular}{@{}l@{}}
                 \(
                    order = 3
                \)
            \end{tabular}
        }\hspace{0.5em} \hfill &
        \resizebox{.16\hsize}{!}{%
            \begin{tabular}{@{}l@{}}
                \(
                    order = 4 - stock
                \)
            \end{tabular}
        } \hspace{0.5em}\hfill &
        \resizebox{.32\hsize}{!}{%
             \begin{tabular}{@{}l@{}}      
                \(
                    order = \begin{cases}
                        8 &\mbox{ if } -13 \leq stock \leq -2 \\
                        2 &\mbox{ otherwise}
                    \end{cases} 
                \)
            \end{tabular}
        } \hspace{0.5em}\hfill &
        \resizebox{.32\hsize}{!}{%
             \begin{tabular}{@{}l@{}}      
                \(
                    order = \begin{cases}
                        6 &\mbox{ if } -10  \leq stock \leq -1 \\
                        1 &\mbox{ if } 2  \leq stock \leq 2 \\
                        3 &\mbox{ otherwise}
                    \end{cases} 
                \)
            \end{tabular}}
    \end{tabular}
    \caption{\textbf{Inventory:} Examples of C, S, PWS1-C and PWS2-C (left to right) policies computed by CGPO.}
    \label{fig:other-policies-inventory}
\end{figure*}

\begin{figure*}[!ht]
    \centering
    \begin{tabular}{cc}
       \resizebox{.16\hsize}{!}{%
            \begin{tabular}{@{}l@{}}
                 \(
                    release_1 = 19.73
                \) \\
                \(
                    release_2 = 55.87
                \)
            \end{tabular}
        } \hfill &
        \resizebox{.3\hsize}{!}{%
            \begin{tabular}{@{}l@{}}
                \(
                    release_1 = level_1 - 20
                \) \\
                 \(
                    release_2 = 0.93 \times level_2 - 29.74
                \)
            \end{tabular}
        }\vspace{0.6em} \\
        \resizebox{.40\hsize}{!}{%
             \begin{tabular}{@{}l@{}}      
                \(
                    release_1 = \begin{cases}
                        1.39 &\mbox{ if } 30.62 \leq level_1 \leq 68.66 \\
                        40.71 &\mbox{ if } 70.93 \leq level_1\leq 100 \\
                        10.622 &\mbox{ otherwise}
                    \end{cases} 
                \) \\
                \(
                    release_2 = \begin{cases}
                        80.78 &\mbox{ if } 159.93 \leq level_2 \leq 200 \\
                        38.69 &\mbox{ if } 144.21 \leq level_2 \leq 159.11 \\
                        18.25 &\mbox{ otherwise}
                    \end{cases}                 
                \)
            \end{tabular}
        } \hspace{0.6em}\hfill &
        \resizebox{.5\hsize}{!}{%
             \begin{tabular}{@{}l@{}}      
                \(
                    release_1 = \begin{cases}
                        level_1 - 50 &\mbox{ if } 20  \leq level_1 \leq 100 \\
                        2 \times level_1 - 100 &\mbox{ otherwise}
                    \end{cases} 
                \) \\
                \(
                    release_2 = \begin{cases}
                        level-2 - 37.8 &\mbox{ if } 38.9 \leq level_2 \leq 200 \\
                        1.17 \times level_2 - 235.5 &\mbox{ otherwise}
                    \end{cases}                 
                \)
            \end{tabular}}
    \end{tabular}
    \caption{\textbf{Reservoir:} Examples of C, S, PWS2-C and PS1-S (left to right, top to bottom) policies computed by CGPO.}
    \label{fig:other-policies-reservoir}
\end{figure*}

\subsection{Ablations}

\begin{figure*}[!htb]
    \centering
    \includegraphics[width=0.33\linewidth,trim={0.5cm 0 0.5cm 0},clip]{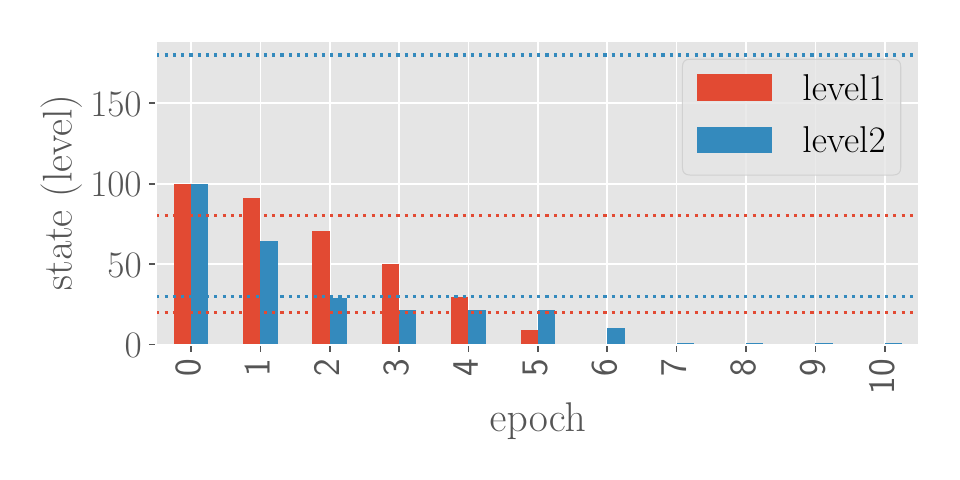}%
    \includegraphics[width=0.33\linewidth,trim={0.5cm 0 0.5cm 0},clip]{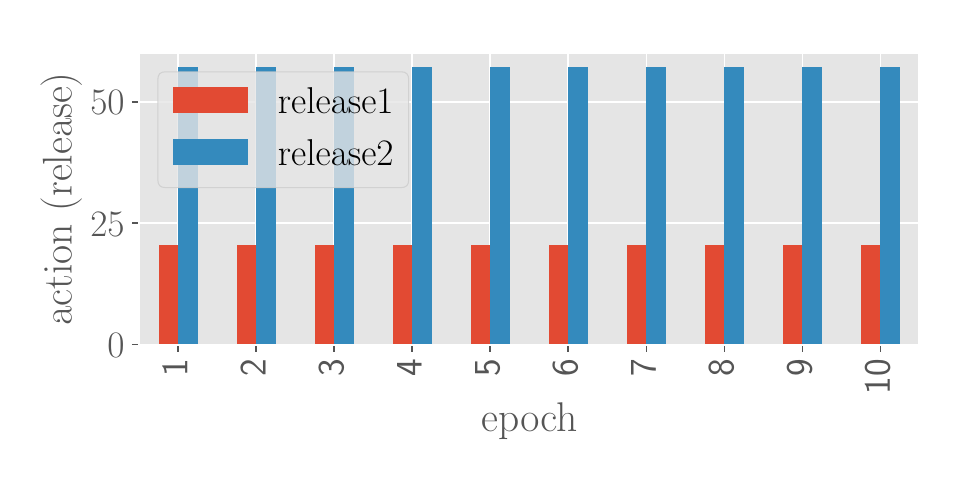}%
    \includegraphics[width=0.33\linewidth,trim={0.5cm 0 0.5cm 0},clip]{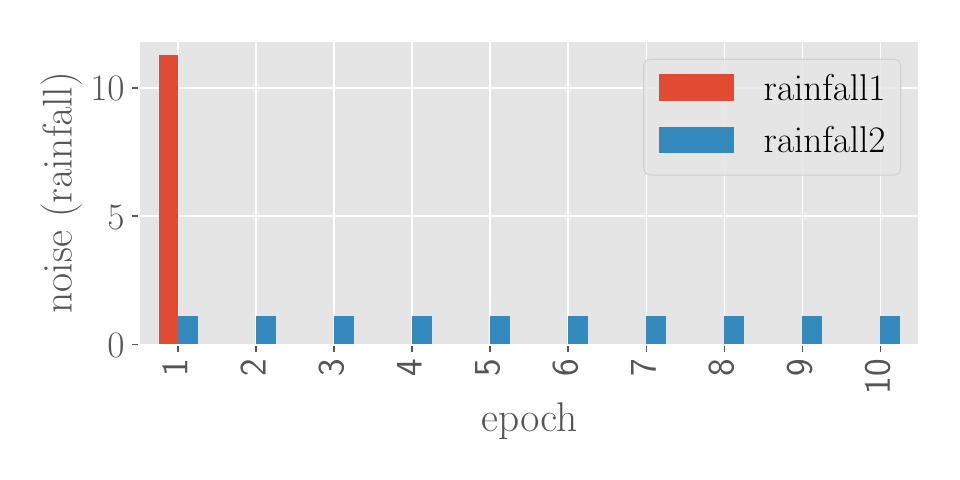} \\
    \includegraphics[width=0.33\linewidth,trim={0.5cm 0 0.5cm 0},clip]{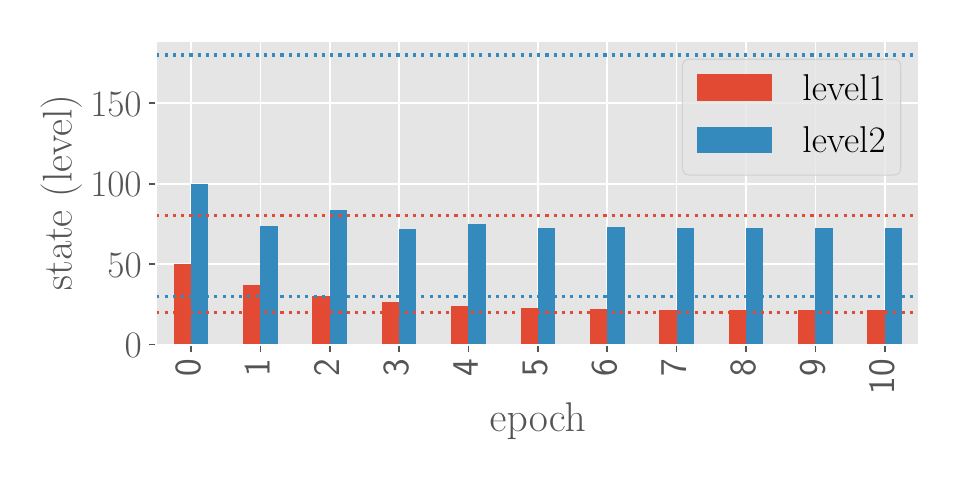}%
    \includegraphics[width=0.33\linewidth,trim={0.5cm 0 0.5cm 0},clip]{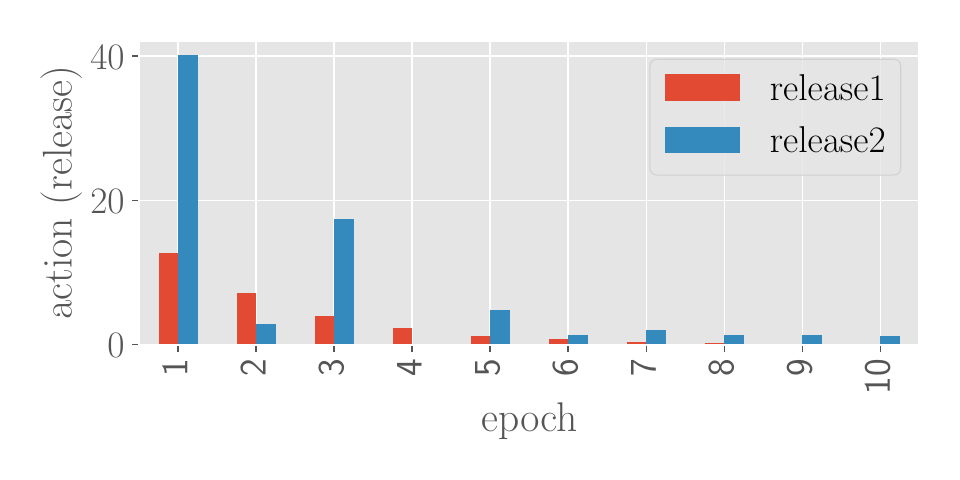}%
    \includegraphics[width=0.33\linewidth,trim={0.5cm 0 0.5cm 0},clip]{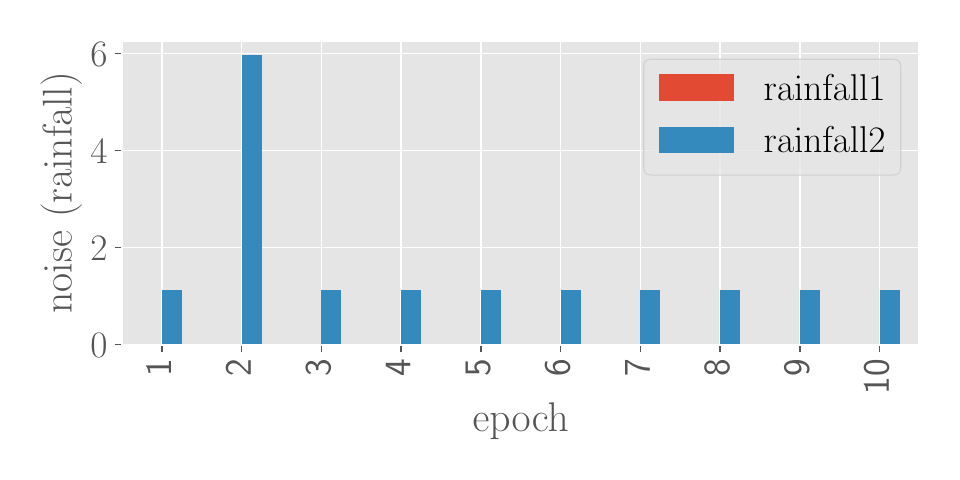}
    \caption{\textbf{Reservoir:} Worst-case state trajectory (left), actions (middle), and noise (right) for C (top) and S policies (bottom) computed by CGPO.  The more expressive S policy provides much more stable water levels.}
    \label{fig:worst}
\end{figure*}

\begin{figure*}[!htb]
    \centering
    \includegraphics[width=0.33\linewidth,trim={0.5cm 0 0.5cm 0},clip]{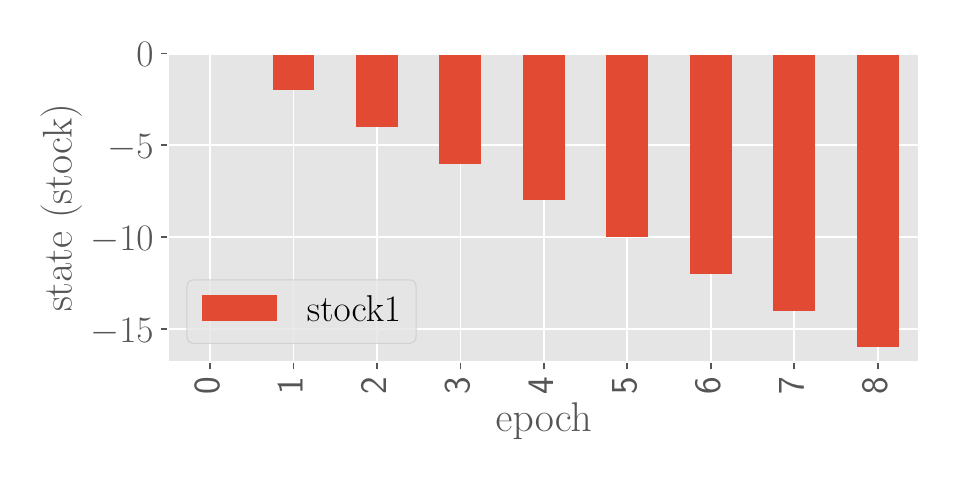}%
    \includegraphics[width=0.33\linewidth,trim={0.5cm 0 0.5cm 0},clip]{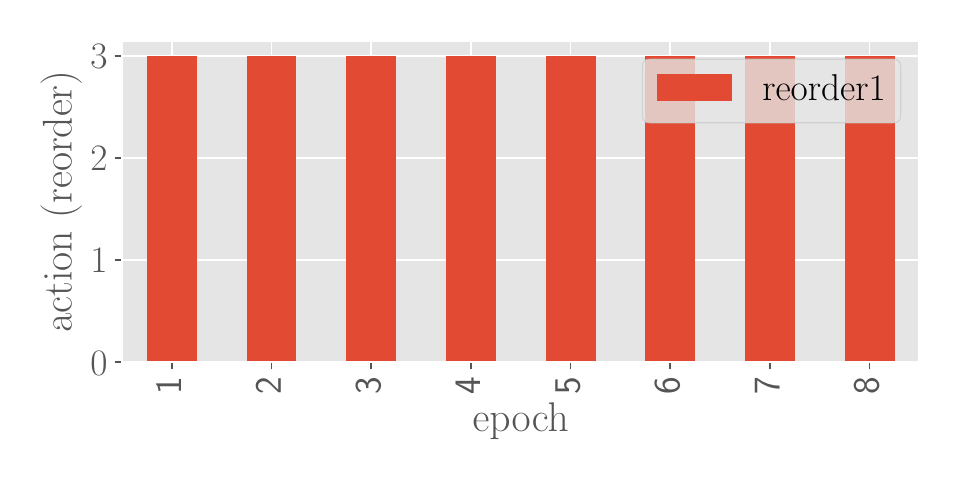}%
    \includegraphics[width=0.33\linewidth,trim={0.5cm 0 0.5cm 0},clip]{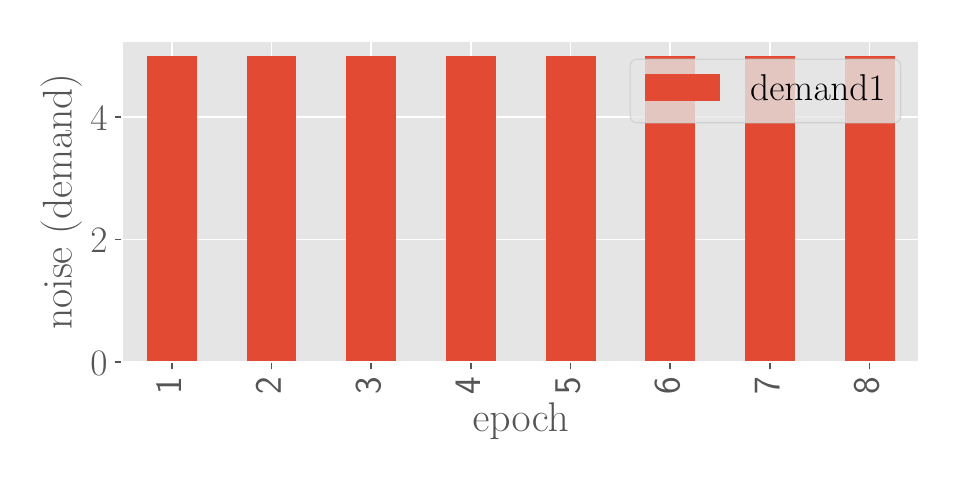}\\
    \includegraphics[width=0.33\linewidth,trim={0.5cm 0 0.5cm 0},clip]{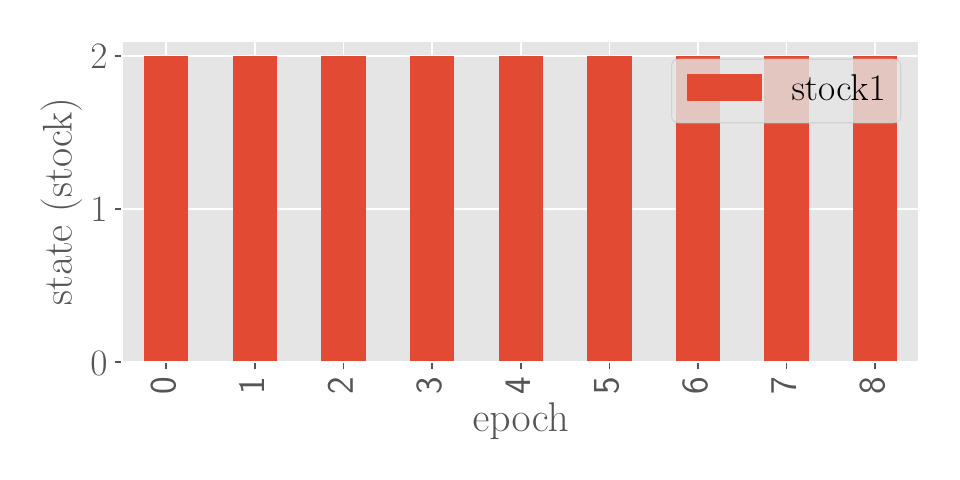}%
    \includegraphics[width=0.33\linewidth,trim={0.5cm 0 0.5cm 0},clip]{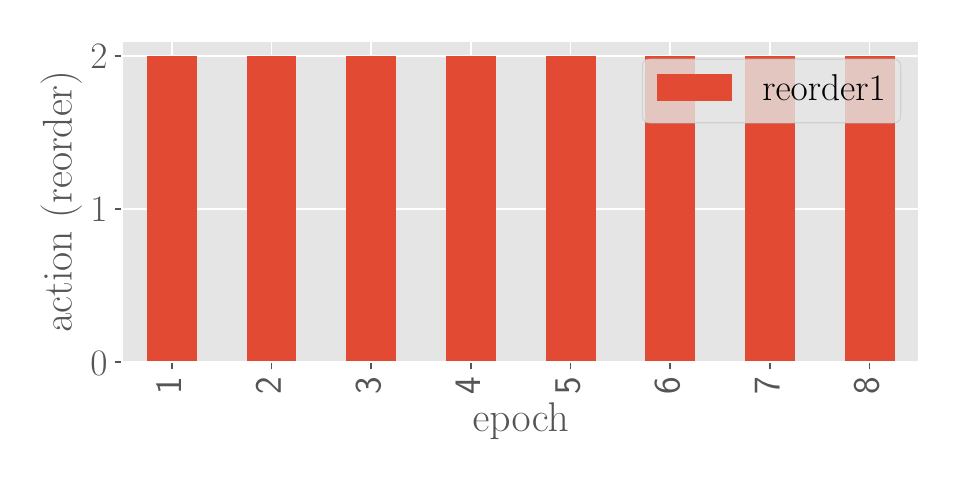}%
    \includegraphics[width=0.33\linewidth,trim={0.5cm 0 0.5cm 0},clip]{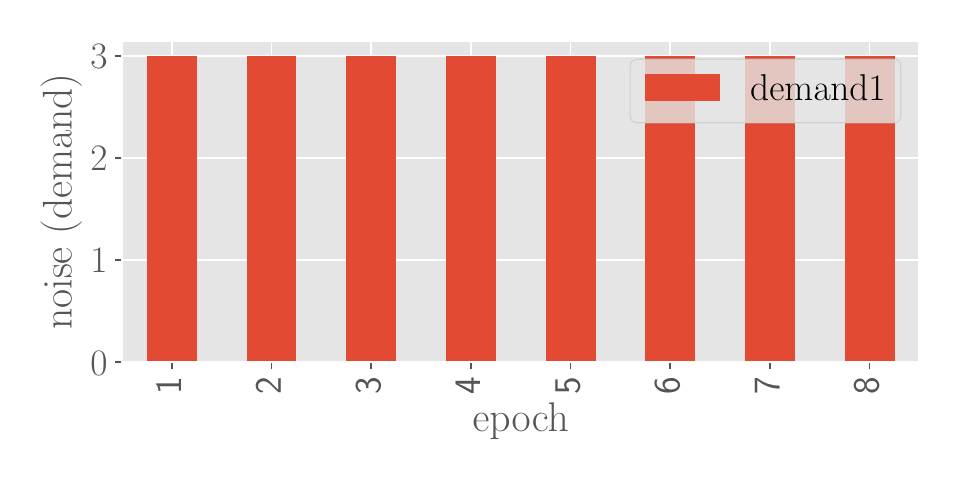}
    \caption{\textbf{Inventory:} Worst-case state trajectory (left), actions (middle), and disturbances/noise (right) for C (top) and S policies (bottom) computed by CGPO.}
    \label{fig:worst-inventory}
\end{figure*}

\begin{figure*}
    \centering
    \includegraphics[width=0.35\linewidth,trim={0.5cm 0 0.5cm 0},clip]{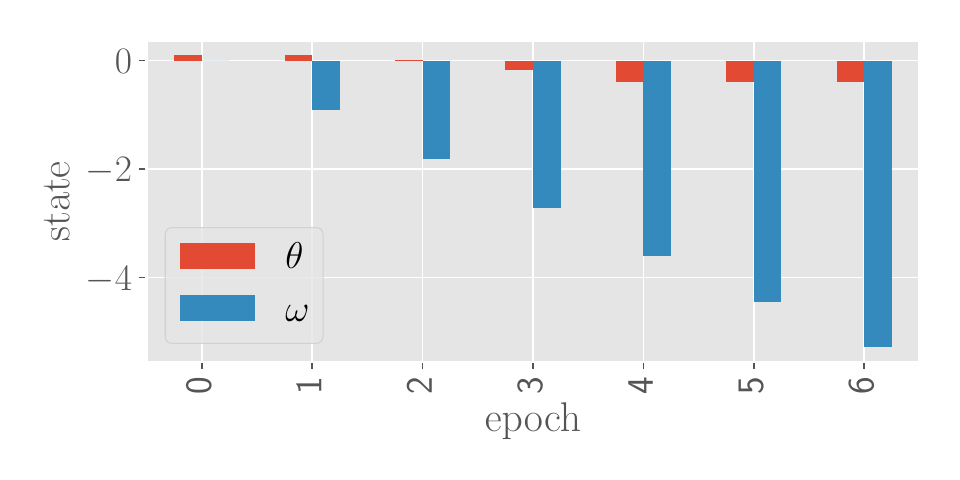}\hspace{0.5em}%
    \includegraphics[width=0.35\linewidth,trim={0.5cm 0 0.5cm 0},clip]{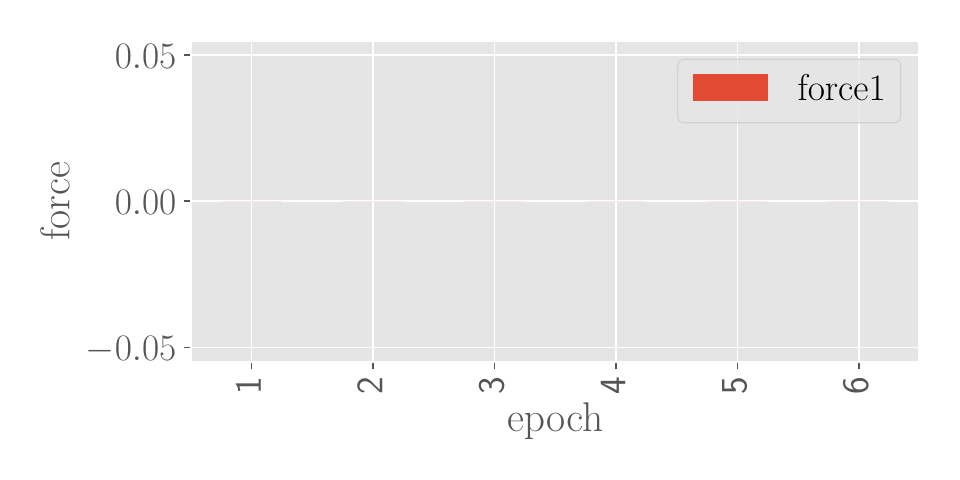}\\
    \includegraphics[width=0.35\linewidth,trim={0.5cm 0 0.5cm 0},clip]{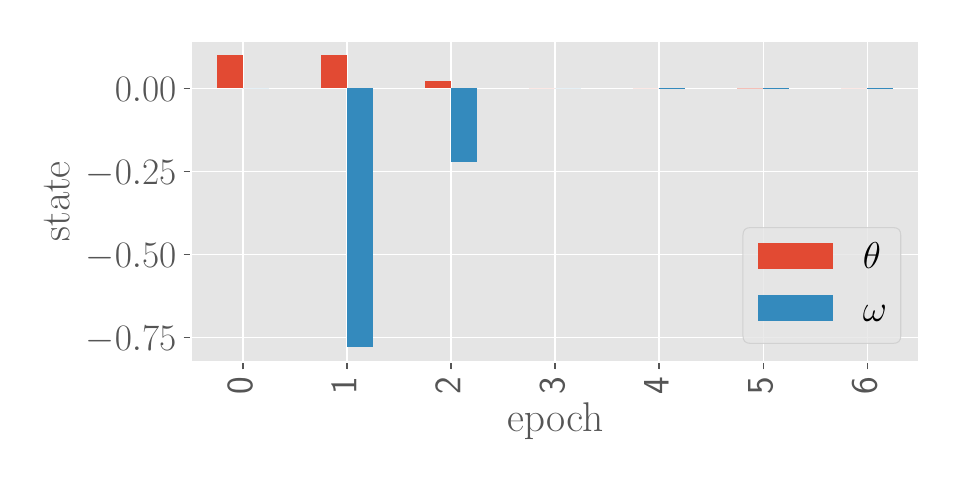}\hspace{0.5em}%
    \includegraphics[width=0.35\linewidth,trim={0.5cm 0 0.5cm 0},clip]{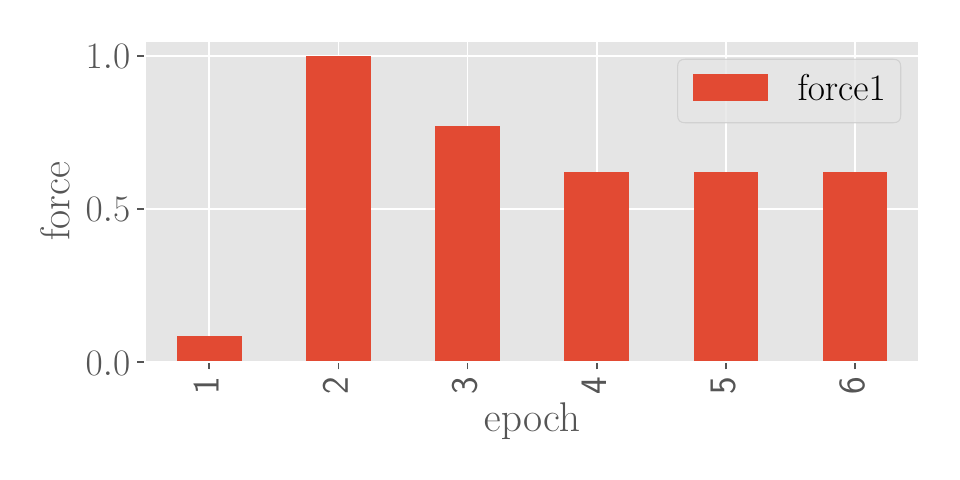}
    \caption{\textbf{VTOL:} Worst-case state trajectory (top), and actions (bottom) for Q policies. First (last) column corresponds to the first (last) iterations of CGPO.}
    \label{fig:worst-vtol}
\end{figure*}

\paragraph{Worst-Case Analysis} Fig. \ref{fig:worst} plots the state trajectory, actions and noise (i.e. rainfall) that lead to worst-case performance of C and S policies after the last iteration of CGPO on Reservoir. Here, we observe that the worst-case performance for the optimal constant-value (C) policy occurs when the rainfall is low and both reservoirs become empty.  Given that the cost of overflow exceeds underflow, this is expected as the C policy must release enough water to prevent high-cost overflow events during high rainfall at the risk of water shortages during droughts. In contrast, the optimal linear (S) policy maintains safe water levels even in the worst-case scenario, since it avoids underflow and overflow events by releasing water (linearly) proportional to the water level. Similar to Reservoir, a non-trivial worst-case scenario is identified by CGPO for Inventory (Fig. \ref{fig:worst-inventory}) for optimal C policies but not S policies. This scenario corresponds to very high demand that exceeds the constant reorder quantity, causing stock-outs. For VTOL (Fig. \ref{fig:worst-vtol}), there is no worst-case scenario with non-zero cost upon convergence at the final iteration of CGPO (bottom row). We have also shown the worst-case scenario computed at the first iteration prior to achieving policy optimality, where a worst-case scenario corresponds to no force being applied causing the system to eventually lose balance.

\begin{table}[!ht]
\centering
\begin{tabular}{ |c||c|c|c|c|  }
 \hline
 & \multicolumn{2}{|c|}{Inventory} & \multicolumn{2}{|c|}{Reservoir} \\
 \hline
 Policy & Inner & Outer & Inner & Outer \\
 \hline
 C & 8 / 121 / 72 & 96 / 673 / 385 & 20 / 0 / 822 & 300 / 0 / 5823 \\
 PWS-S & 32 / 129 / 72 & 384 / 772 / 387 & 80 / 0 / 842 & 1200 / 0 / 7133 \\
 \hline
 C & 160 / 0 / 64 & 780 / 0 / 480 & 540 / 0 / 320 & 4035 / 0 / 2400 \\
 PWS-S & 168 / 0 / 112 & 709 / 168 / 1056 & 560 / 0 / 440 & 3797 / 540 / 4200 \\
 \hline
\end{tabular}
 \caption{Number of binary/integer/real-valued decision variables in the inner and outer optimization problems at the last iteration of CGPO (top half), and the number of linear/quadratic/general constraints (bottom half).}
\label{table:scale}
\end{table}

\paragraph{Analysis of Problem Size}
To understand how the problem sizes in CGPO scale across differing expressivity classes of policies and dynamics, Table \ref{table:scale} shows the number of decision variables and constraints in the MIPs at the last iteration of constraint generation.
For VTOL with a quadratic (Q) policy, the number of variables in the inner and outer problems are 0 / 0 / 212 and 0 / 0 / 512, respectively, while the number of constraints are 114 / 18 / 72 and 270 / 95 / 155. The number of decision variables/constraints are fixed in the inner problem and grow linearly with iterations in the outer problem.  
The number of variables and constraints do not grow significantly with the policy class expressiveness. 

\section{Conclusion}
We presented CGPO, a novel bilevel mixed-integer programming formulation for DC-MDP policy optimization in predefined expressivity classes with {bounded optimality guarantees}.  {\bf To the best of our knowledge, CGPO is the first algorithm capable of providing bounded error guarantees for a broad range of DC-MDPs and policy classes.}  We used constraint generation to decompose CGPO into an inner problem that produces an adversarial worst-case constraint violation (i.e., trajectory), and an outer problem that 
improves the policy w.r.t.\ these trajectories. 
Across diverse domains and policy classes, we showed that the learned policies (some provably optimal) and their worst-case performance/failure modes were easy to interpret, while maintaining a manageable number of variables and constraints. One challenge with CGPO lies in the significant computational demands needed to tackle larger problems. To overcome this limitation, future work should focus on approximating the solutions to the inner and outer problems by leveraging tools from reinforcement learning.

\section{Appendix}

In this supplement to the main paper, we provide further experimental details for successful implementation of CGPO. We also provide a worked example illustrating the mechanics of CGPO, additional convergence/performance results and worst-case analysis for the domains that were left out of the main text due to space limitations. We conclude with a discussion of related work.

\subsection{Worked Example}
\label{sec:app-worked-example}

\paragraph{Domain Description} Let the state $s_t$ denote the continuous position of a particle on $\mathbb{R}$ at epoch $t = 1, 2$, and let the action be a continuous unbounded displacement $a_t$ applied to the particle, so that the state updates according to $s_{t+1} = s_t + a_t$. Thus, $\states = \actions = \reals$ and suppose further that $\states_1 = [0, 5]$. If the target position is designated as $10$, we formulate the reward as $r(s_{t+1}) = -|s_{t+1} - 10|$, and thus the value function is $V(a_1, s_1) = -|s_1 + a_1 - 10|$. 

\paragraph{The Outer Problem} For ease of illustration, we focus our attention on {linear} policies $\pi_{w,b}$ of the form $a_t = b + w s_t$, where $(w, b)$ are real-valued decision variables. The goal is to find the optimal policy parameters:
\begin{equation*}
    \min_{(w, b) \in \mathbb{R}^2} \max_{s_1 \in [0, 5]} \max_{a_1 \in \mathbb{R}} \, [V(a_1, s_1) - V((w, b), s_1)],
\end{equation*}
where $V((w, b), s_1) = -|s_1 + b + w s_1 - 10|$ is the value of the linear policy. However, this problem is not directly solvable due to the inner $\max$, but it can be written as
\begin{equation*}
    \min_{(w, b) \in \mathbb{R}^2,\, \varepsilon \in [0, \infty)} \varepsilon \qquad \textrm{s.t.} \quad \varepsilon \geq V(a_1, s_1) - V((w, b), s_1),
\end{equation*}
where constraints are enumerated by all possible $(s_1, a_1)$ pairs. Since the number of constraints is infinite for continuous states or actions, it suggests a constraint generation approach, where a large set of diverse $(s_1, a_1)$ pairs are used to ``build up" the constraint set one at a time, decoupling the overall problem into a sequence of solvable MIPs. However, a key question arises: \textbf{{which} state-action pairs $(s_1, a_1)$ should be considered for inclusion into the outer problem?} While many different reasonable choices exist, a particularly logical choice is to select the state-action pairs on which the policy $\pi_{w,b}$ parameterized by $(w,b)$ performs {worst}.

\paragraph{Solve the Outer Problem} The outer problem thus maintains a finite subset of constraints from the infinitely-constrained problem. We begin with an arbitrary $(s, a)$-pair, which forms the first constraint of the outer problem. For illustration, we select $(s_1, a_1) = (0, 0)$, thus with the constraint 
\begin{equation*}
    \varepsilon \geq V(0, 0) - V((w, b), 0) = -10 + |b - 10|,
\end{equation*}
the outer problem becomes:
\begin{equation*}
    \min_{w, b \in \mathbb{R},\, \varepsilon \in [0, \infty)} \, \varepsilon \qquad \mathrm{s.t.} \quad \varepsilon \geq -10 + |b - 10|.
\end{equation*}
Recognizing that the minimum of the r.h.s. of the constraint is attained at $b^* = 10$, and $w^*$ arbitrary, we select $(w^*, b^*) = (0, 10)$ as the solution of the outer problem with optimal $\varepsilon = 0$. Next, we will try to improve upon the policy by adding an additional constraint to the outer problem.

\paragraph{Solve the Inner Sub-Problem} To accomplish this, we need to add a new scenario $(s_1, a_1)$ in relation to which the linear policy with parameters $(w, b) = (0, 10)$ performs worst. Thus, we solve:
\begin{align*}
    \max_{s_1 \in [0, 5], a_1 \in \reals} [ V(a_1, s_1) - V((w, b), s_1) ] = \max_{s_1 \in [0, 5], a_1 \in \reals} [ -|s_1 + a_1 - 10| + s_1],
\end{align*}
from which we find $(s_1^*, a_1^*) = (5, 5)$. The optimal value of this problem is positive, so it represents a potentially binding constraint in the outer problem. Thus, we could improve upon the policy by adding the corresponding constraint.

\paragraph{Solve the Outer Problem Again} Adding the constraint to the outer problem:
\begin{equation*}
        \min_{(w, b) \in \reals^2,\, \varepsilon\in[0, \infty)} \quad \varepsilon \qquad \textrm{s.t.} \quad  \varepsilon \geq -10 + |b - 10|, \quad \varepsilon \geq |b + 5 w - 5|,
\end{equation*}
whose new solution can be found as $b^* = 10, w^* = -1$ with optimal $\varepsilon = 0$. Finally, the worst-case scenario for this new policy can be found from the inner sub-problem; one solution is $(s_1^*, a_1^*) = (0, 10)$ with value zero. This corresponds to a non-binding constraint, and thus we terminate with the optimal linear policy $\pi^*(s_1) = 10 - s_1$.

\subsection{Linear Policy + Linear Dynamics Result in a Polynomial Class}
\label{sec:app-linear-case}

Suppose both the policy and transition functions are linear in the state. Specifically, consider the transition function $s' = s + a$, and the policy $a = w s$. Then, defining $s$ the initial state, $s'$ the immediate successor state following $s$, and $s''$ the immediate successor state following $s'$:
\begin{align*}
    s' &= s + a = s + w s = (w + 1) s \\
    a' &= w s' = w (w + 1) s = O(w^2) \\
    s'' &= s' + a' = (w + 1) s + w (w + 1) s = (w + 1)^2 s \\
    a'' &= w s'' = w (w + 1)^2 s = O(w^3),
\end{align*}
which are polynomial in $w$. 

\subsection{The Nonlinear Intercept Problem}
\label{sec:app-intercept}

Intercept is a nonlinear domain with Boolean actions. It involves two objects moving in continuous trajectories in a subset of $\mathbb{R}^2$, in which one object (e.g. missile, bird) flies in a parabolic arc across space towards the ground, and must be intercepted by a second object (e.g. anti-ballistic missile, predator) that is fired from a fixed position on the ground. While the problem is best described in continuous time, we study a discrete-time version of the problem. The state includes the position $(x_t, y_t)$ of the missile at each decision epoch, as well as ``work" variables indicating whether the interceptor has already been fired $f_t$ as well as its vertical elevation $i_t$. Meanwhile, the action $a_t$ is Boolean-valued and indicates whether the interceptor is fired at a given time instant. 

The state transition model can be written as:
\begin{align*}
    x_{t+1} &= x_t + v_x, \qquad y_{t+1} = h - \frac{1}{2}g x_t^2 \\
    f_{t+1} &= f_t \lor a_t, \qquad i_{t+1} = \begin{cases}
        i_t &\mbox{ if } f_t = 0 \\
        i_t + v_y &\mbox{ if } f_t = 1
    \end{cases},
\end{align*}
ignoring the gravitational interaction of the interceptor. Interception happens whenever the absolute differences between the coordinates of the missile and the interceptor are within $\delta$, so the reward is
\begin{equation*}
    r(s_{t+1}) = \begin{cases}
        1 &\mbox{ if } f_{t+1} \land | x_{t+1} - i_x | \leq \delta \land | y_{t+1} - i_{t+1} | \leq \delta \\
        0 &\mbox{ otherwise}.
    \end{cases}.
\end{equation*}
We set $v_x = 0.1, h = 5, g = 9.8, i_x = 0.7$. 

We study Boolean-valued policies with linear constraints (i.e., B, PWL-B) that take into account the position of the missile at each epoch, i.e. PWL1-B policies of the form
\begin{equation*}
    a_t(x_t, y_t) = \begin{cases}
        a_1 &\mbox{ if } l \leq w_1 x_t + w_2 y_t \leq u \\
        a_2 &\mbox{ otherwise}
    \end{cases}
\end{equation*}
where $a_1, a_2 \in \lbrace 0, 1 \rbrace$ and $l, u, w_1, w_2$ are all tunable parameters.

\subsection{Additional Implementation Details}
\label{sec:app-details}

\paragraph{Computing Environment} We use the Python implementation of the Gurobi Optimizer (GurobiPy) version 10.0.1, build v10.0.1rc0 for Windows 64-bit systems, with an academic license. Default optimizer settings have been used, with the exception of NumericFocus set to 2 in order to enforce numerical stability, and the MIPGap parameter set to 0.05 to terminate when the optimality gap reaches 5\%.  To compute the tightest possible bounds on decision variables in the MIP compilation, we use interval arithmetic \citep{hickey2001}. All experiments were conducted on a single machine with an Intel 10875 processor (base at 2.3 GHz, overclocked at 5.1 GHz) and 32 GB of RAM. A runtime cap of 2 hours was enforced for each run. However, we do not report runtimes due to the highly stochastic and unpredictable nature of the runs.

\paragraph{Domain Description Files}
Domains are described in \emph{Relational Dynamic influence Diagram Language} (RDDL), a structured planning domain description language particularly well-suited for modelling problems with stochastic effects \citep{sanner2010}. To facilitate experimentation, we wrote a general-purpose routine for compiling RDDL code into a Gurobi mixed integer program formulation using the \texttt{pyRDDLGym} interface \citep{taitler2022}. 

\paragraph{Action Constraints} One important issue concerns how to enforce constraints on valid actions. Two valid approaches include (1) explicitly constraining actions in RDDL constraints (state invariants and action preconditions) by compiling them as MIP constraints, or (2) implicitly clipping actions in the state dynamics (conditional probability functions in RDDL). Crucially, we found the latter approach performed better during optimization, since constraining actions inherently limits the policy class to a subset of weights that can only produce legal actions across the initial state space. This not only significantly restricts policies (i.e. for linear valued policies, the weights would be constrained to a tight subset where the output for every possible state would be a valid action) but also poses challenges for the optimization process, which is significantly complicated by these action constraints. 

\paragraph{Encoding Constraints in Gurobi} Mathematical operations, such as strict inequalities, cannot be handled explicitly in Gurobi. To perform accurate translation of such operations in our code-base, we used indicator/binary variables. For instance, to model the comparison $a > b$ for two numerical values $a, b$, we define a binary variable $y \in \lbrace 0, 1 \rbrace$ and error parameter $\epsilon > 0$ constrained as follows:
\begin{align*}
    y == 1 \implies a \geq b + \epsilon, \qquad 
    y == 0 \implies a \leq b.
\end{align*}
In practice, $\epsilon$ is typically set larger than the smallest positive floating point number in Gurobi (around $10^{-5}$), but is often problem-dependent. We derived optimal policies for all domains using $\epsilon = 10^{-5}$, except Intercept where we needed to use a larger value of $\epsilon = 10^{-3}$ to allow Gurobi to correctly distinguish between the two cases above.

\subsection{Additional Experimental Results}
\label{sec:app-more-results}

\begin{figure*}[!htb]
    \centering
    \includegraphics[width=0.45\linewidth]{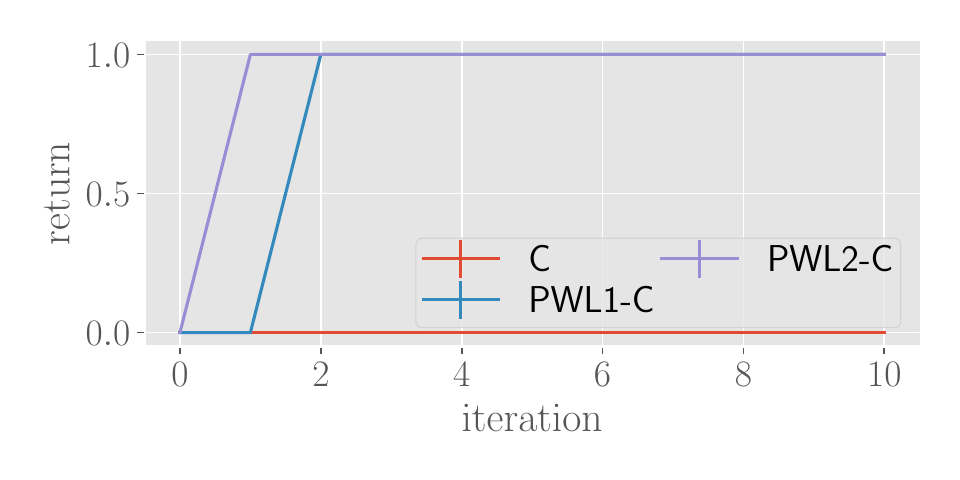}%
    \includegraphics[width=0.45\linewidth]{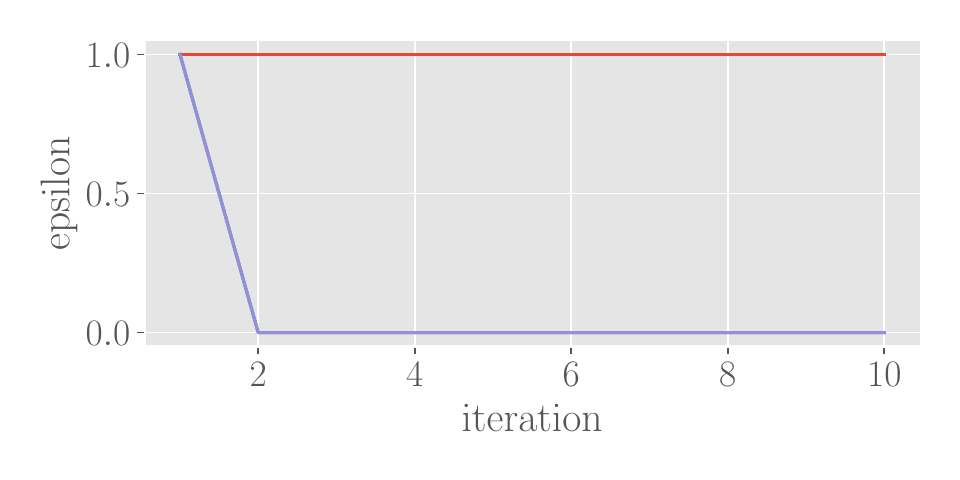}
    \caption{\textbf{Intercept:} Simulated return (left) and optimal value of $\varepsilon^*$ (right) as a function of the number of iterations of constraint generation.}
    \label{fig:missile}
\end{figure*}

\begin{figure*}[!htb]
    \centering
    \begin{tabular}{ccc}
        \resizebox{.09\hsize}{!}{%
            \begin{tabular}{@{}l@{}}
                 \(
                    fire = 1
                \)
            \end{tabular}
        } \hspace{0.5em} \hfill &
        \resizebox{.4\hsize}{!}{%
            \begin{tabular}{@{}l@{}}
                 \(
                    fire = \begin{cases}
                        0 &\mbox{ if } -0.136 \leq -0.2 x - 0.02 y \leq 0 \\
                        1 &\mbox{ otherwise}
                    \end{cases}  
                \)
            \end{tabular}
        } \hspace{0.5em} \hfill &
        \resizebox{.4\hsize}{!}{%
            \begin{tabular}{@{}l@{}}
                 \(
                    fire = \begin{cases}
                        0 &\mbox{ if } -0.136 \leq -0.2 x - 0.02 y \leq 0 \\
                        0 &\mbox{ if } 0.001 \leq -x \leq 1 \\
                        1 &\mbox{ otherwise}
                    \end{cases}  
                \)
            \end{tabular}
        }
    \end{tabular} \\
    \includegraphics[width=0.3\linewidth]{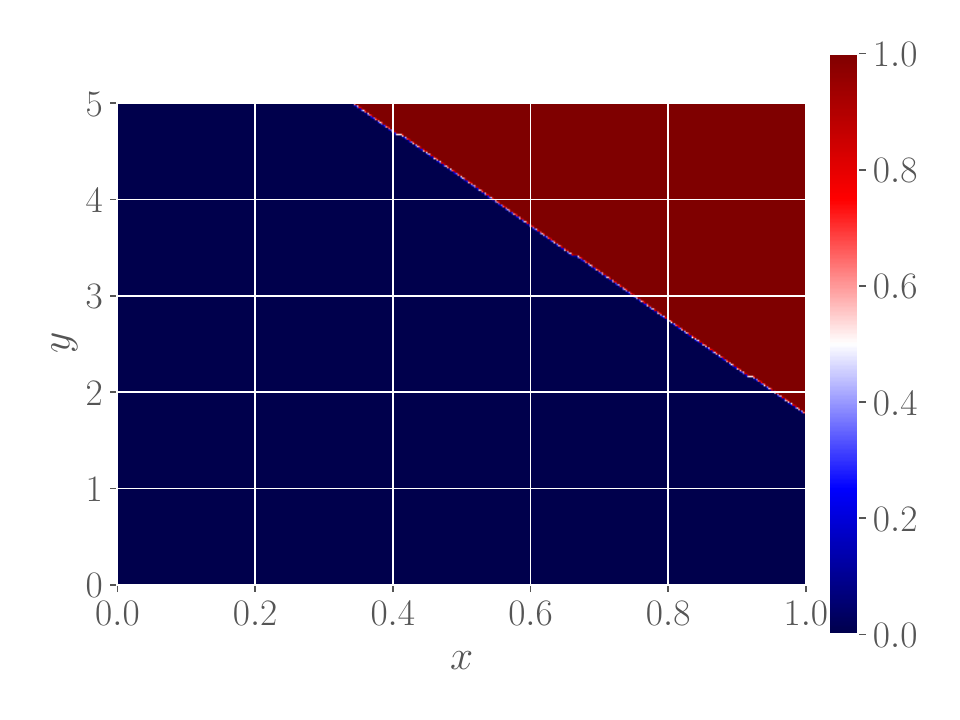}
    \caption{\textbf{Intercept:} Examples of B, PWS1-B and PWS2-B policies (left to right) computed by CGPO, and visualization of PWS-B policy (bottom).}
    \label{fig:intercept-policies}
\end{figure*}

\paragraph{Intercept} Fig. \ref{fig:missile} illustrates the return and worst-case error of Boolean-valued policies for Intercept. Only the policies with at least one case are optimal, with corresponding error of zero. As illustrated in the last plot in Fig. \ref{fig:intercept-policies}, the optimal policies fire whenever a threat is detected in the top right corner of the airspace. 

\paragraph{Analysis of Problem Size} Fig. \ref{fig:size-inventory}, Fig. \ref{fig:size-reservoir} and Fig. \ref{fig:size-vtol} summarize the total number of variables and constraints in the outer MIP formulations solved by CGPO at each iteration for Inventory, Reservoir and VTOL, respectively. Each iteration of constraint generation requires a roll-out from the dynamics and reward model, which in turn requires instantiate a new set of variables to hold the intermediate state and reward computations, and thus the number of variables and constraints grows linearly with the iterations. Even at the last iteration of CGPO, we see that the number of variables and constraints remains manageable.

\begin{figure*}
    \centering
    \includegraphics[width=0.45\linewidth]{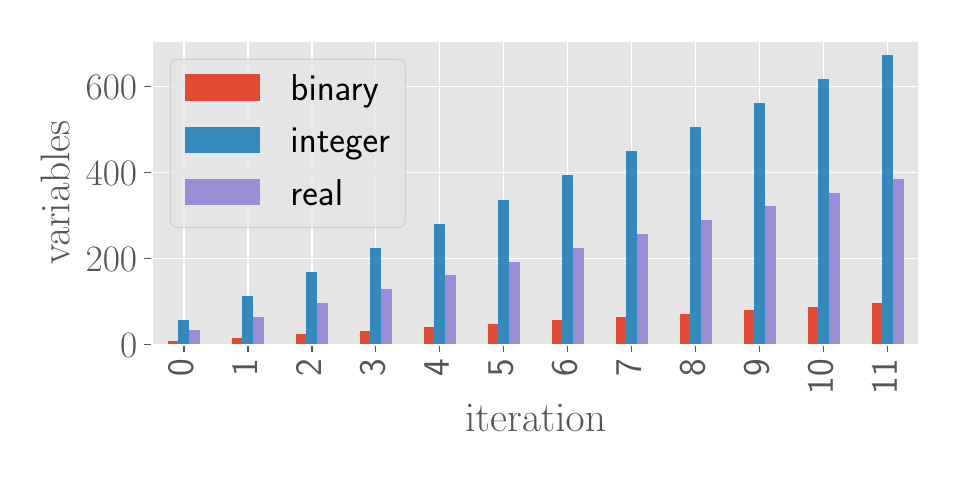}%
    \includegraphics[width=0.45\linewidth]{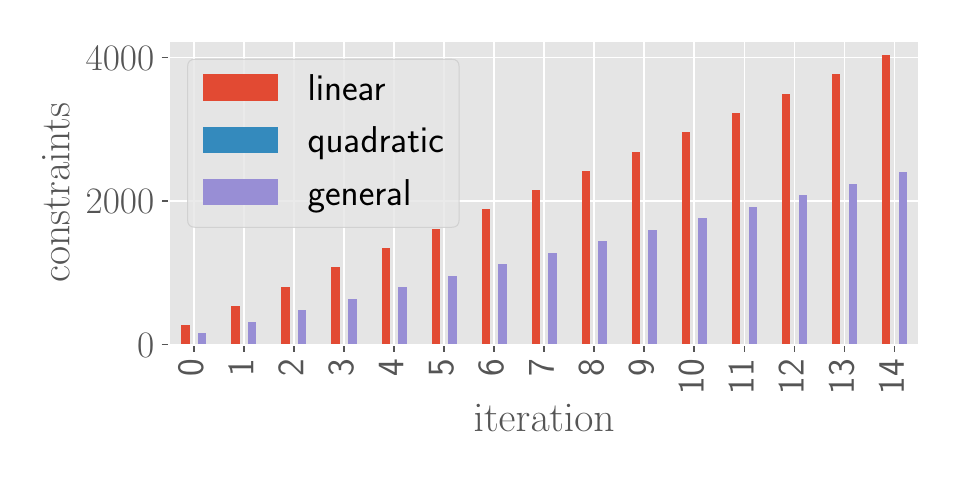} \\
    \includegraphics[width=0.45\linewidth]{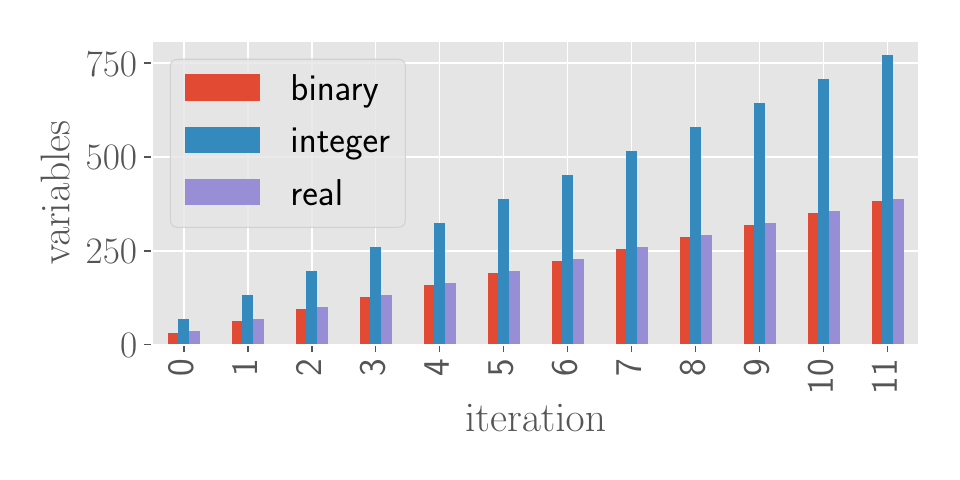}%
    \includegraphics[width=0.45\linewidth]{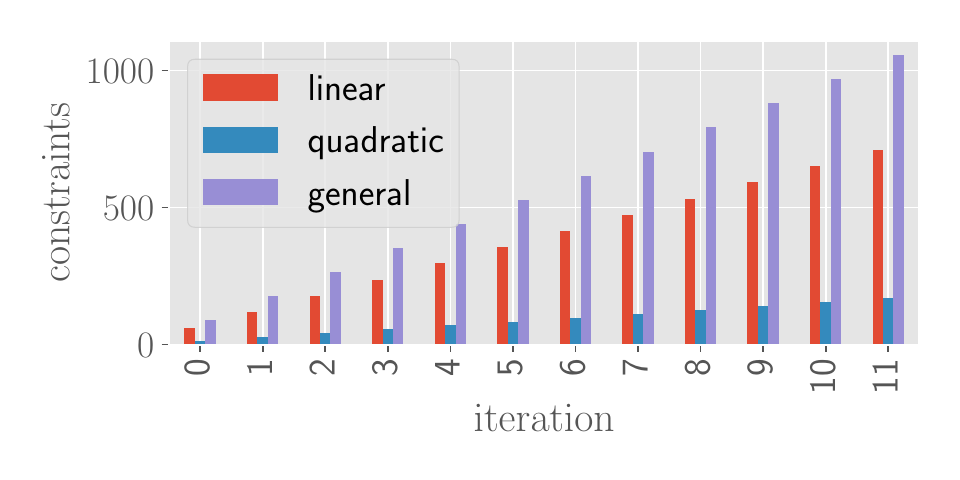}
    \caption{\textbf{Inventory:} Number of decision variables (left) and constraints (right) for C policy (top) and PWS1-S policy (bottom) corresponding to the outer optimization problem at each iteration of CGPO.}
    \label{fig:size-inventory}
\end{figure*}

\begin{figure*}
    \centering
    \includegraphics[width=0.45\linewidth]{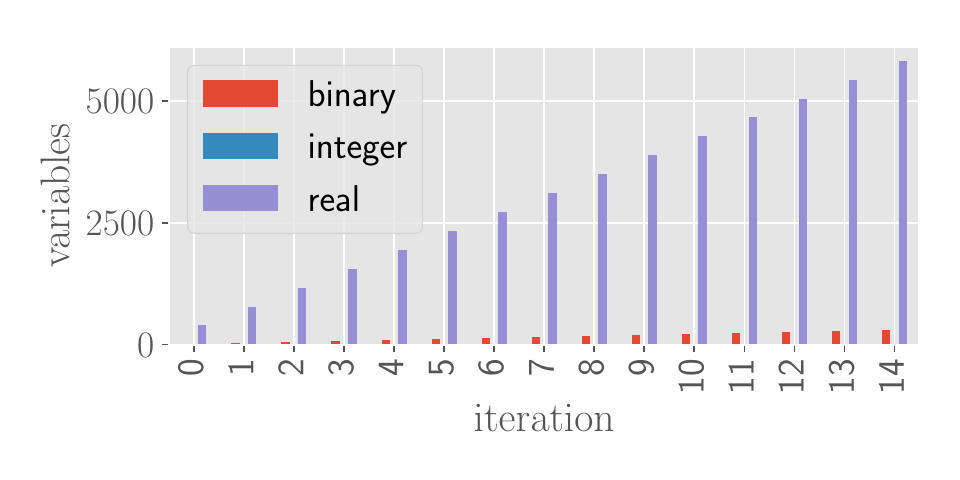}%
    \includegraphics[width=0.45\linewidth]{figures/scalability/Reservoirlinear1100995Couterconstraints.pdf} \\
    \includegraphics[width=0.45\linewidth]{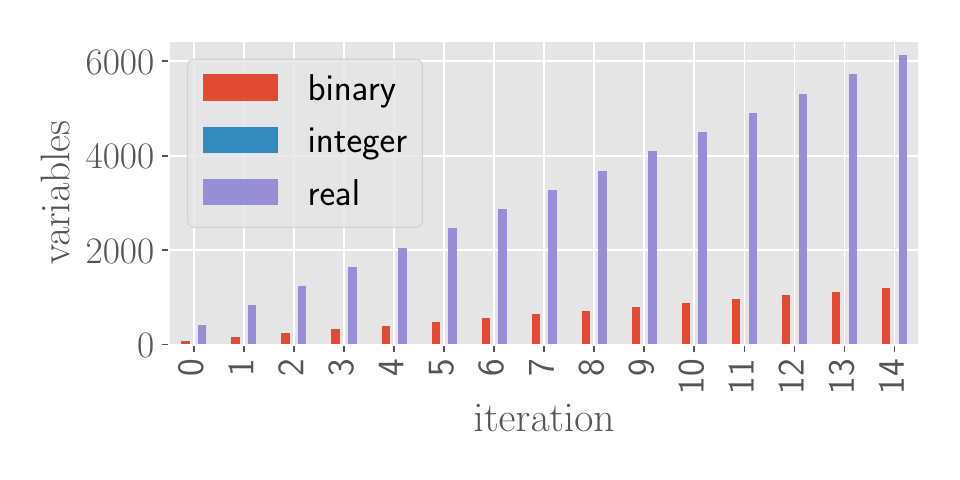}%
    \includegraphics[width=0.45\linewidth]{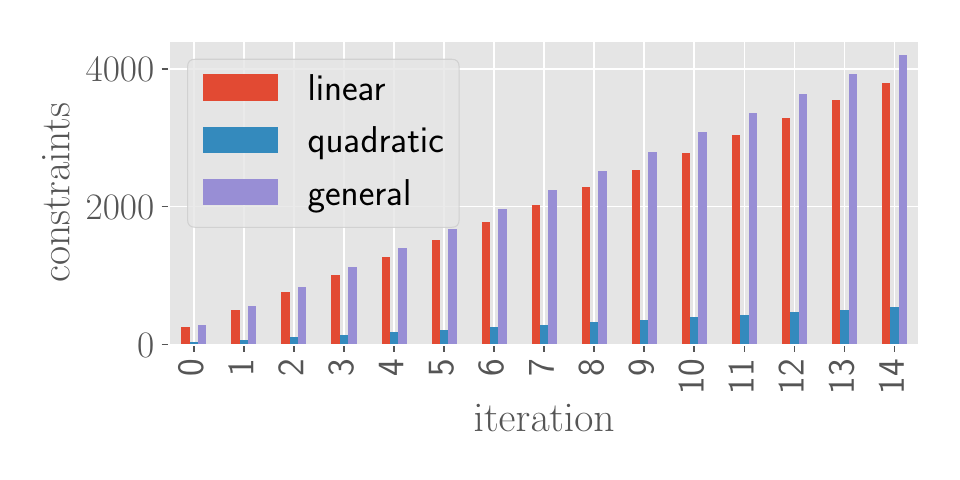}
    \caption{\textbf{Reservoir:} Number of decision variables (left) and constraints (right) for C policy (top) and PWS1-S policy (bottom) corresponding to the outer optimization problem at each iteration of CGPO.}
    \label{fig:size-reservoir}
\end{figure*}

\begin{figure*}
    \centering
    \includegraphics[width=0.45\linewidth]{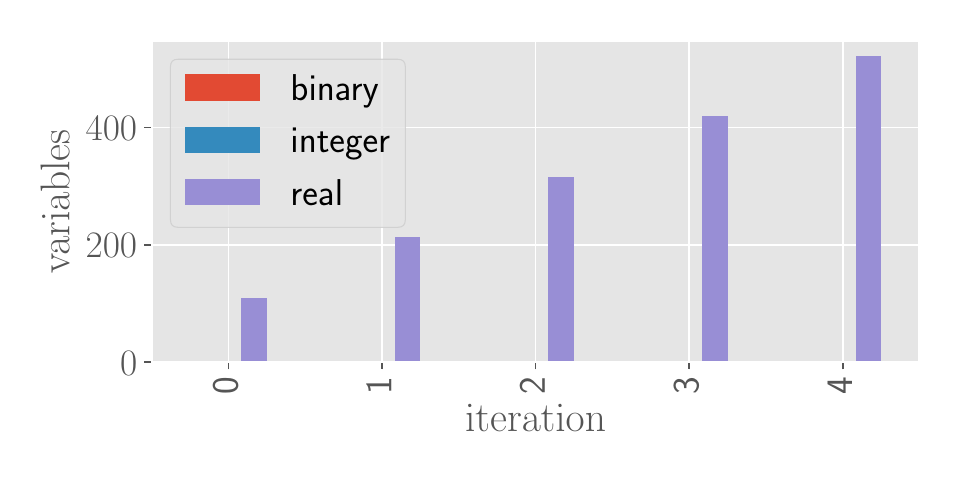}%
    \includegraphics[width=0.45\linewidth]{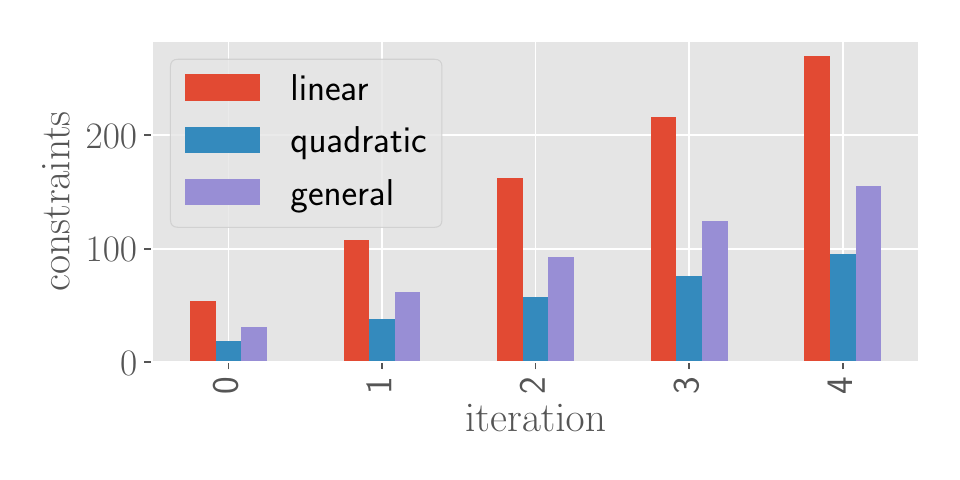}
    \caption{\textbf{VTOL:} Number of decision variables (left) and constraints (right) for Q policy corresponding to the outer optimization problem at each iteration of CGPO.}
    \label{fig:size-vtol}
\end{figure*}

\subsection{Additional Related Work}
\label{sec:app-related-work}

The scheme of mixed discrete-continuous models can be traced back to the hybrid automata \citep{henzinger1996theory}, and how to verify reachability and model checking \citep{henzinger1998algorithmic}. 
In \citep{ferretti2014recent} a policy iteration approach has been taken to synthesize a feedback controller to a continuous-time system with discrete jumps. \emph{Model predictive control} (MPC) approaches are especially appealing in the hybrid setting for controller synthesis 
\citep{borrelli2017predictive}, however they are not employed for worst-case analysis and policy optimization as in this work.

\paragraph{(Chance-) Constrained Optimization in Hybrid Systems}
The tool of constrained optimization is popular in policy optimization for hybrid systems \citep{borrelli2003constrained}. For instance, \citep{SATO2021217} posed the problem of conjunctive synthesis and system falsification as constrained optimization. In the probabilistic or uncertain case, chance-constraints are popular for finding a robust policy; 
\citep{5477242} used predictive control to synthesis the optimal controller in expectation under chance constraints.
\citep{petsagkourakis2022chance} utilized an recurrent neural network as parameterized policy for policy optimization under chance-constraints. Our work differs from the above as we optimize a given policy structure, under the worst case, thus providing guarantees on the bounded policy optimality for the chosen policy class (also providing interpretability of the final optimized policy, due to the compactness of the policy class).

\paragraph{Constrained Policy Optimization in Reinforcement Learning} A different stream of work incorporates safety constraints on parameterized policies \citep{achiam2017constrained,liu2021policy,polosky2022constrained}. However, they typically only guarantee convergence through gradient-based methods, and not bounded policy optimality as in our work. Furthermore, to accomplish this, they require differentiable -- and often non-compact policy classes (i.e. such as neural network) -- which makes them unsuitable for directly optimizing piecewise policy classes or other compact policy classes with discrete structure as studied in this work.

\bibliographystyle{unsrtnat}
\bibliography{references}

\end{document}